\documentclass[a4paper, 11pt, reqno]{amsart}
\usepackage[margin=2.3cm,vmargin= 4.3cm]{geometry}
\usepackage{graphicx}
\usepackage{amsmath}
\usepackage{amsthm,amsfonts,amssymb,mathrsfs,amscd,amstext,amsbsy}
\usepackage[all]{xy}
\usepackage{tikz}
\usepackage[english]{babel}
\usepackage{enumitem}
\usepackage{hyperref}
\usepackage{amsthm,leftidx}
\usepackage{subcaption}
\usepackage{arydshln}
\usepackage{multirow}
\usepackage{mathtools}

\hypersetup{
    pdftitle=   {Some techniques to compute a galois cohomology of an algebraic tori},
   pdfauthor=  {Seung Kyun Park}
}

\renewcommand{\theequation}{\thesection.\arabic{equation}}
\newcommand{\MyQuote}[1]{\vspace{.4cm}\refstepcounter{equation}%
    \hspace{0.2cm} \parbox{14.4cm}{#1}\hfill (\theequation)\vspace{.4cm}}
\newcommand{\norm}[3]{\text{N}_{#1 / #2}#3}
\newcommand{\brau}[2]{\text{Br}(#1|#2)}
\newcommand{\coho}[3]{H^{#1}(#2,#3)}
\newcommand{\spec}[1]{\text{Spec}\, #1}
\newcommand{\ind}[3]{\text{Ind}_{#1}^{#2}(#3)}

\newtheorem{thm}{Theorem}[section]
\newtheorem{prop}[thm]{Proposition}
\newtheorem{lem}[thm]{Lemma}
\newtheorem{exa}[thm]{Example}
\newtheorem{cor}[thm]{Corollary}
\newtheorem{remark}[thm]{Remark}

\newtheorem{defn}[thm]{Definition}
\numberwithin{equation}{section}

\title{An algorithm for the classification of twisted forms of toric varieties}
\author{Seungkyun Park}
\address{Department of Mathematical Sciences, KAIST, 291 Daehak-ro Yuseong-gu Daejeon, 34141 South Korea}
\email{phcmas@kaist.ac.kr}

\begin{document}

\begin{abstract}
Let $K/k$ be a finite Galois extension, $G=\text{Gal}(K/k)$, $\Sigma$ be a fan in a lattice $N$ and $X_{\Sigma}$ be an associated toric variety over $k$. It is well known that the set of $K/k$-forms of $X_{\Sigma}$ is in bijection with $\coho{1}{G}{\text{Aut}_{\Sigma}^T}$, where $\text{Aut}_{\Sigma}^T$ is an algebraic group of toric automorphisms of $X_{\Sigma}$. In this paper, we suggest an algorithm to compute $\coho{1}{G}{\text{Aut}_{\Sigma}^T}$ and find that followings can be classified via this algorithm : $K/k$-forms of all toric surfaces, $K/k$-forms of all 3-dimensional affine toric varieties with no torus factor, $K/k$-forms of all 3-dimensional quasi-projective toric varieties when $K/k$ is cyclic.

\end{abstract}
\maketitle

\section{Introduction}
Let $k$ be an arbitrary field. In this paper, an algebraic variety over $k$ means a separated, geometrically integral scheme of finite type over $k$. Throughout this paper, let $K/k$ be a finite Galois extension, $G=\text{Gal}(K/k)$ be a Galois group and $X_K = X\otimes_{k} K$ be a scalar extension of a variety $X$. The following problem has been of interest for a long time.
\begin{quote}
Let $X$ be an algebraic variety over $k$. Can we classify all isomorphism classes of varieties $Y$ over $k$ such that $Y_K\simeq X_K$?
\end{quote}

The vareity $Y$ is called a $K/k$-form of $X$ or a twisted form of $X$. When $X$ is quasi-projective, the set of $K/k$-forms of $X$ is in bijection with $\coho{1}{G}{\text{Aut}(X_K)}$(\cite{serre}-III or Section 2.5). Hence, classifying $K/k$-forms of $X$ is reduced to compute $\coho{1}{G}{\text{Aut}(X_K)}$. Without any restriction of $X$, it is very difficult to compute $\coho{1}{G}{\text{Aut}(X_K)}$. We begin to observe when $X$ is a toric variety.

Let $\mathcal{T}$ be an algebraic torus over $k$ that splits over $K$. All $n$-dimensional algebraic tori can be classified by conjugacy classes of group homomorphisms $\varphi: G \rightarrow \text{Aut}_{\Sigma}\subset \text{GL}(n,\mathbb{Z})$. Denote by ${}_{\varphi}\mathcal{T}$ the torus corresponding to $\varphi$ and denote by ${}_{\varphi}\mathcal{T}(K)$ its scalar extension ${}_{\varphi}\mathcal{T}\otimes_k K$. A pair $(Y,\mathcal{T})$ is called an arithmetic toric variety over $k$ if $Y$ is a normal variety over $k$ and it has faithful $\mathcal{T}$-action which has a dense open orbits. When $\mathcal{T}$ is split, $X$ is called a split toric variety. All split toric varieties are classified by a fan $\Sigma$. Denote by $X_{\Sigma}$ the split toric variety over $k$ associated to $\Sigma$. Since any arithmetic toric variety is a twisted form of a split toric variety, the study of an arithmetic toric variety is reduced to the study of a twisted form of a split toric variety. 

In 2013, Elizondo, Lima-Filho, Sottile and Teitler studied twisted forms of split toric varieties \cite{elizondo}. Let $\text{Aut}_{\Sigma}^T$ be an algebraic group of toric automorphisms of $X_{\Sigma}$. With a result of Wlodarczyk \cite{wlodarczyk}, they proved that when $\Sigma$ is quasi-projective or $K/k$ is quadratic, the set of $K/k$-forms of $X_{\Sigma}$ is in bijection with $\coho{1}{G}{\text{Aut}_{\Sigma}^T}$. They succeed to write this Galois cohomology set in terms of $\coho{1}{G}{{}_{\varphi}\mathcal{T}(K)}$ and computed $\coho{1}{G}{{}_{\varphi}\mathcal{T}(K)}$ for all 2-dimensional torus ${}_{\varphi}\mathcal{T}$. Based on this computation, they classified real forms of all toric surfaces. Huruguen recently studied\cite{huruguen} the compactification of spherical orbits which is more general than arithmetic toric varieties. A spherical orbit of a connected reductive algebraic group $G$ over $k$ is a pair $(X_0,x_0)$, where $X_0$ is a homogeneous space for $G$ on which a Borel subgroup of $G$ has a dense orbit and $x_0$ is a $k$-rational point of $X_0$. Huruguen characterized situations where spherical embeddings over $k$ are classified by a combinatorial object called a Galois-stable colored fan. Moreover, he constructed an example of a smooth toric variety over $k$ under a 3-dimensional nonsplit torus, which admits no twisted form. Arithmetic toric varieties also play an important rule to study nonsplit tori via smooth projective compactifications. This work begans with Brylinski\cite{brylinski} who suggested the construction of a complete projective fan $\Sigma$ in a lattice $N$ that is invariant under the action of $G$ on $N$. In 1982, Voskresenski\u\i\cite{voskresenskii} showed there exist a smooth toric variety $Y$ over $k$ such that $Y_K \simeq X_{\Sigma}$ associated to that fan. Using this, Batryrev and Tschinkel\cite{batyrev} studied the distribution of rational points of bounded height on compactifications of nonsplit tori.

In this paper, we generalize the techniques of \cite{elizondo}. We suggest an algorithm to compute $H^1(G,$ $\text{Aut}_{\Sigma}^T)$ which is applicable regardless of the dimension of $X_{\Sigma}$. Whether this algorithm is applicable depends on the conjugacy class of $\text{Aut}_{\Sigma} \subset \text{GL}(n,\mathbb{Z})$, where $\text{Aut}_{\Sigma}$ is a subgroup of $\text{GL}(n,\mathbb{Z})$ preserving a fan $\Sigma$. For $n=2$ and $n=3$, Appendix A summarizes conjugacy classes of finite subgroups of $\text{GL}(n,\mathbb{Z})$ that can be applied this algorithm. All conjugacy classes of $\text{GL}(2,\mathbb{Z})$ and 61 cases among 73 conjugacy classes of $\text{GL}(3,\mathbb{Z})$ belongs to Appendix A. 

It is well-known that $\coho{1}{G}{\text{Aut}_{\Sigma}^T}$ is the same as a disjoint union of the quotients of $\coho{1}{G}{{}_{\varphi}\mathcal{T}(K)}$ by some subgroup of $\text{GL}(n,\mathbb{Z})$. For a more detailed description, see (\cite{elizondo},Theorem 3.4) or Theorem \ref{decent of toric} of this paper. Computing $\coho{1}{G}{{}_{\varphi}\mathcal{T}(K)}$ is thus the first step to obtain $\coho{1}{G}{\text{Aut}_{\Sigma}^T}$. In Section 3, we introduce theorems to compute $\coho{1}{G}{{}_{\varphi}\mathcal{T}(K)}$ and to find its generators. There are three main theorems of Section 3: Theorems \ref{induced 1}, \ref{induced 2} and \ref{main thm of 3.3}. These transform the problem of computing $\coho{1}{G}{{}_{\varphi}\mathcal{T}(K)}$ into the problem of checking the similarity of matrices over $\mathbb{Z}$. 

In Section 4, we construct an algorithm to compute $\coho{1}{G}{\text{Aut}_{\Sigma}^T}$. As an example, we show the application of the algorithm for the split toric variety $X_{\Sigma}$ with $\text{Aut}_{\Sigma}\simeq D_6$. And we compute the number of $K/k$-forms of $X_{\Sigma}$ for various field extension $K/k$. Finally, we discuss which $K/k$-forms of which toric varieties can be classified via this algorithm. We find that the followings can be classified via this algorithm. Besides the following lists, we expect that there exists other cases that can be applied this algorithm.
\begin{itemize}
\item[$\cdot$] $K/k$-forms of all toric surfaces for any finite Galois extension $K/k$
\item[$\cdot$] $K/k$-forms of all 3-dimensional quasi-projective toric varieties when $K/k$ is a cyclic extension 
\item[$\cdot$] $K/k$-forms of all 3-dimensional affine toric varieties with no torus factor for any finite Galois extension $K/k$.
\end{itemize}

\textbf{Acknowledgement.} I would like to thank my advisor Sanghoon Baek for his valuable comments. This work has been supported by a TJ Park Science Fellowship of POSCO TJ Park Foundation and a grant from the National Research Foundation of Korea(NRF) funded by the Ministry of Science ICT and Future Planning(2016R1C1B2010037).

\section{Preliminaries}

\subsection{Nonabelian Cohomology}
Let $G$ be a group and $A$ be a $G$-group. The 0-th nonabelian cohomology group is defined by $\coho{0}{G}{A} = A^{G} = \{a \in A \, | \, g\cdot a = a \,\,\text{for all}\,\, g\in G\}$, where $\cdot$ denotes the $G$-action on $A$. A 1-coycle is a map $g \mapsto c_g$ from $G$ to $A$ satisfying $c_{gh}=c_g \, (g\cdot c_h)$. Two 1-cocycles $c$ and $c^{\prime}$ are called cohomologous if there exists $a \in A$ such that $c^{\prime}_g = a^{-1}c_g (g \cdot a)$ for all $g \in G$. This is an equivalence relation on cocycles and we write $\coho{1}{G}{A}$ for the set of equivalence classes. The first nonabelian cohomology is a pointed set $\coho{1}{G}{A}$ with a distingushed element $\mathbf{1}_{g}=1$ for all $g \in G$. Let $B$ be a $G$-group and $A$ be a normal subgroup of $B$. If the $G$-action on $B$ preserves a normal subgroup $A$, then we have a long exact sequence as follows. 
\begin{equation}\label{les module}
1 \rightarrow \coho{0}{G}{A} \rightarrow \coho{0}{G}{B} \rightarrow \coho{0}{G}{B/A} \overset{\delta}\rightarrow
\coho{1}{G}{A} \rightarrow \coho{1}{G}{B} \rightarrow \coho{1}{G}{B/A}
\end{equation}
For $c \in (B/A)^G$, choose $b \in B$ wih $c=bA$. We define $\delta(c)$ as a 1-cocycle that sends $g$ to $b^{-1}(g\cdot b)$. When $A$ is abelian, the nonabelian cohomology is same as the usual group cohomology. Let $\cdots \rightarrow F_{i} \overset{d_{i+1}}\rightarrow \cdots \rightarrow F_0 \overset{d_1} \rightarrow \mathbb{Z}\rightarrow 0$ be the bar resolution as described in (\cite{weibel}, Chapter 6), here $F_i = \mathbb{Z}G \otimes \cdots \otimes \mathbb{Z}G$ ($i+1$ times). By taking $\text{Hom}(,A)$, we have a chain
$0 \rightarrow \text{Hom}_{\mathbb{Z}G}(F_0,A) \overset{d_1\,}\rightarrow \text{Hom}_{\mathbb{Z}G}(F_1,A) \overset{d_2\,}\rightarrow \cdots$. The $n$-th cohomology of this chain is denoted by $\coho{n}{G}{A}$. If $N$ is a normal subgroup of $G$, we get an exact sequence as follows.
\begin{equation}\label{les group}
1\longrightarrow \coho{1}{G/N}{A^N} \overset{\text{inf}}\longrightarrow \coho{1}{G}{A} \overset{\text{res}}\longrightarrow \coho{1}{N}{A}^{G/N} \overset{\text{tg}}\longrightarrow \coho{2}{G/N}{A^N} \overset{\text{inf}}\longrightarrow \coho{2}{G}{A}
\end{equation}

Consider when $G\simeq \mathbb{Z}/2\mathbb{Z}=\{e,g\}$ where $e$ is an identity. Let $\mathcal{U}\simeq \mathbb{Z}$ be an alternating $\mathbb{Z}/2\mathbb{Z}$-module. For any $\mathbb{Z}/2\mathbb{Z}$-module $A$, we have
\begin{align} \label{alternating}
\coho{i}{\mathbb{Z}/2\mathbb{Z}}{\mathcal{U} \otimes_{\mathbb{Z}} A} \simeq \coho{i+1}{\mathbb{Z}/2\mathbb{Z}}{A} \quad \text{for all $i\geq 0$}
\end{align}
This can be proved by applying $\text{Hom}(\cdot,\mathcal{U}\otimes A) \simeq \text{Hom}(\cdot \otimes \mathcal{U},A)$ to the free resolution 
\[\cdots \overset{e-g} \longrightarrow \mathbb{Z}[\mathbb{Z}/2\mathbb{Z}]  \overset{e+g} \longrightarrow \mathbb{Z}[\mathbb{Z}/2\mathbb{Z}] \overset{e-g}\longrightarrow \mathbb{Z}[\mathbb{Z}/2\mathbb{Z}] \longrightarrow \mathbb{Z} \longrightarrow 0\]

Now, we concentrate on when $G = \text{Gal}(K/k)$ where $K$ is a finite Galois extension over $k$. Suppose $G$ acts on $K^\times$ as a Galois group. Then, $\coho{1}{G}{K^\times}$ is trivial and it is called Hilbert's Theorem 90. And we have $\coho{2}{G}{K^\times} \simeq \brau{k}{K}$ (See \cite{central}, Section 4.4), where $\brau{k}{K}$ is a group of similarity classes of central simple $k$-algebras that splits over $K$. When $K/k$ is a cyclic extension, this relation gives $ \brau{k}{K} \simeq k^{\times}/\text{Im}\,\norm{K}{k}$, where $\norm{K}{k}$ is a Galois norm of $K/k$.

\subsection{Split Toric Scheme and Toric Automorphism}
This subsection is devoted to introduce the definition of toric scheme and toric automorphism refered to (\cite{elizondo}, Section 2.1, 2.2). Let $N$ be a lattice of finite rank with dual lattice $M = \text{Hom}(N,\mathbb{Z})$. A semigroup of the form $\mathbb{N}\mathcal{A}$ for some subset $\mathcal{A} \subset N$ is called an affine semigroup of $N$. A cone $\sigma$ is a finitely generated saturated affine  semigroup of $N$. The dual of $\sigma$ is defined by
$\sigma^{\vee} = \{ u \in M \, | \, u(v) \geq 0 \,\, \text{for all} \, v \in \sigma \}$. When 0 is a face of a cone $\sigma$, it is called strongly convex.

Each cone $\sigma \subset N$ gives an affine toric scheme $U_{\sigma} = \spec{\mathbb{Z}[\sigma^{\vee}]}$ where $\mathbb{Z}[\sigma^{\vee}]$ is a semigroup ring. A fan $\Sigma$ in $ N$ is a finite collection of strongly convex cones satisfying follows.
\begin{itemize}
\item[1.] For any $\sigma \in \Sigma$, each face of $\sigma$ is also in $\Sigma$.
\item[2.] For all $\sigma_{1},\sigma_{2} \in \Sigma$, the intersection $\sigma_{1} \cap \sigma_{2}$ is a face of each.
\end{itemize}
Under the condition of the fan, we can glue $U_{\sigma}$'s for all $\sigma \in \Sigma$ and get a scheme $X$ over $\mathbb{Z}$. Let $k$ be an arbitrary base field. Denote by $X_{\Sigma}$ the scheme $X\otimes_{\mathbb{Z}}k$. An affine algebraic group $\mathbb{T}_N = \spec{k[M]}$ is called a split torus over $k$. The split torus $\mathbb{T}_N$ acts on $U_{\sigma,k} = U_{\sigma}\otimes_{\mathbb{Z}} k$ by the ring map $k[\sigma^{\vee}]\rightarrow k[\sigma^{\vee}] \otimes_{k} k[M]$ that sends $u$ to $u\otimes u$. By glueing these morphisms, $\mathbb{T}_{N}$-action can be extended to $X_{\Sigma}$.

A toric automorphism of $X_{\Sigma}$ is a pair $(\varphi,\alpha)$ where $\varphi$ is a group automorphism of $\mathbb{T}_{N}$, $\alpha$ is an automorphism of $X_{\Sigma}$ that makes following diagram commutes. Since $\alpha$ maps torus orbits to torus orbits, $\varphi$ belongs to $\text{Aut}_{\Sigma}$ which is a subgroup of $\text{Aut}(N)$ that preserves a fan $\Sigma$.
\[\xymatrix{\mathbb{T}_{N} \times_{\text{Spec}\,k} X_{\Sigma} \ar[d]_{(\varphi,\alpha)}\ar[r]& X_{\Sigma} \ar[d]_{\alpha}\\
\mathbb{T}_{N} \times_{\text{Spec}\,k} X_{\Sigma} \ar[r]& X_{\Sigma}}
\]
Let $\mathbb{T}_N(k)$ be a set of $k$-valued points. Now, define an algebraic group $\text{Aut}_{\Sigma}^T$ whose $k$-valued point is $\text{Aut}_{\Sigma}^T(k) = \mathbb{T}_N(k) \rtimes \text{Aut}_{\Sigma}$ with a multiplication $(t_1,\psi)(t_2,\varphi)=(t_1 \psi(t_2), \psi \, \circ \, \varphi)$. The $G$-action on $K$-valued points $\text{Aut}_{\Sigma}^T(K)$ is defined as $g \cdot (t,\varphi) = (g(t),\varphi)$ for all $g \in G$ and $\varphi \in \text{Aut}_{\Sigma}$.

\subsection{Descent Theory}
An arithmetic toric variety over $k$ is a pair $(Y,\mathcal{T})$, where $\mathcal{T}$ is a torus over $k$ and $Y$ is a normal variety over $k$ with a faithful action of $\mathcal{T}$ which has a dense orbits. A morphism from $(Y_1,\mathcal{T}_1)$ to $(Y_2,\mathcal{T}_2)$ is defined as a pair $(\varphi,\alpha)$, where $\varphi:\mathcal{T}_1 \rightarrow \mathcal{T}_2$ is a group morphism and $\alpha :Y_1 \rightarrow Y_2$ is a morphism of varieties that preserves a torus action. Let $K$ be a finite Galois extension of $k$ such that $\mathcal{T}$ splits over $K$. Then, $(Y_K,\mathcal{T}_K)$ is isomorphic to $(X_{\Sigma,K},\mathbb{T}_{N,K})$ for some fan $\Sigma$ in $N$. Therefore, the study of arithmetic toric varieties over $k$ reduced to the study of $K/k$-forms of split toric varieties. In this subsection, we keep the notation :  $G=\text{Gal}(K/k)$, $N$ is a lattice of rank $n$ and $\Sigma$ is a fan in $N$.

The set of isomorphism classes of $K/k$-forms of an $n$-dimensional split torus is in bijection with the set of conjugacy classes of group homomorphisms $\varphi: G \rightarrow \text{GL}(n,\mathbb{Z})$. Now, let ${}_{\varphi}\mathcal{T}$ denotes an algebraic torus over $k$ corresponds to a group homomorphism $\varphi:G \rightarrow \text{GL}(n,\mathbb{Z})$. As a $G$-module, its scalar extension ${}_{\varphi}\mathcal{T}(K)={}_{\varphi}\mathcal{T}\otimes_{\mathbb{Z}} K$ is isomorphic to $K^{\times n}$. For $g \in G$, $(k_1,\cdots,k_n) \in {}_{\varphi}\mathcal{T}(K)$ and $\varphi(g) = (a_{ij}) \in \text{GL}(n,\mathbb{Z})$, its $G$-action is described as follows.
\begin{align}\label{G-action}
g \cdot(k_{1},\cdots,k_{n}) &= ( \prod_{1\leq i \leq n} g(k_i)^{a_{1i}},\cdots, \prod_{1 \leq i \leq n} g(k_i)^{a_{ni}} )
\end{align}

\begin{exa}
In this example, we classify $\mathbb{C}/\mathbb{R}$-forms of $\mathbb{C}^{\times}$. Let $G =\operatorname{Gal}(\mathbb{C}/\mathbb{R}) = \{e,g \}$ where $e$ is an identity, $g$ is a conjugation and $\varphi: G \rightarrow \operatorname{Aut}(\mathbb{Z})=\{1,-1\}$ be a group homomorphism. If $\varphi(g)=-1$, the $G$-action on $\mathbb{C}^{\times}$ is given by $g \cdot a = \overline{a}^{-1}$ where $\overline{\phantom{a}}$ denotes conjugation. And we have $(\mathbb{C}^{\times})^{G} = S^1$. Similarly, we get $\mathbb{R}^\times$ when $\varphi$ is an identity. So, $S^1$ and $\mathbb{R}^\times$ are all 1-dimensional real tori.
\end{exa}

Now, let me introduce an action of $\text{GL}(n,\mathbb{Z})$ on ${}_{\varphi}\mathcal{T}(K)$. For any $T=(a_{ij}) \in \text{GL}(n,\mathbb{Z})$ and $(k_1,\cdots,k_n) \in {}_{\varphi}\mathcal{T}(K)$, define
\begin{align} \label{matrix action}
T * (k_{1},\cdots,k_{n}) &= ( \prod_{1\leq i \leq n} k_i^{a_{1i}},\cdots, \prod_{1 \leq i \leq n} k_i^{a_{ni}} )
\end{align} 
Take any 1-cocycle $c$ which is a representation of an element of $\coho{1}{G}{{}_{\varphi}\mathcal{T}(K)}$. Define a $\text{GL}(n,\mathbb{Z})$-action on $\coho{1}{G}{{}_{\varphi}\mathcal{T}(K)}$ by $(T * c)_g = T * c_g$ for all $g\in G$.

\begin{thm}\label{decent of toric}(\cite{elizondo}, Theorem 3.2, Theorem 3.4)
Let $\varphi:G\rightarrow \operatorname{Aut}_{\Sigma}\subset \operatorname{GL}(n,\mathbb{Z})$ be a group homomorphism and let $C_{\operatorname{Aut}_{\Sigma}}(\varphi(G))$ be a centralizer of the image $\varphi(G) \subset \operatorname{Aut}_{\Sigma}$ that acts on $\coho{1}{G}{{}_{\varphi}\mathcal{T}(K)}$ as the previous paragraph. Denote by $\coho{1}{G}{\leftidx{_\varphi}{\mathcal{T}}(K)}/C_{\operatorname{Aut}_{\Sigma}}(\varphi(G))$ the set of its orbits.  Suppose $K/k$ is a quadratic extension or $\Sigma$ is quasiprojective.  Then,
\begin{enumerate}
\item[(1)] $\coho{1}{G}{\operatorname{Aut}_{\Sigma}^T}$ is in bijection with the set of isomorphism classes of $K/k$-forms of $(X_{\Sigma},\mathbb{T}_N)$.
\item[(2)] $\coho{1}{G}{\operatorname{Aut}_{\Sigma}^{T}} \simeq  \coprod_{\varphi} \coho{1}{G}{\leftidx{_\varphi}{\mathcal{T}}(K)}/C_{\operatorname{Aut}_{\Sigma}}(\varphi(G))$ where $\varphi$ varies over representatives of conjugacy classes of $\varphi$.
\end{enumerate}
\end{thm}

\begin{prop}(\cite{elizondo}, Proposition 3.7)\label{computation reduce}
Let $L$ be an intermediate field $k \subset L \subset K$ which is fixed by $\ker{\varphi}$ and $\overline{\varphi}: G/\ker{\varphi} \simeq \operatorname{Gal}(L/k) \rightarrow \operatorname{Aut}_{\Sigma}\subset \operatorname{GL}(n,\mathbb{Z})$ be an injective group homomorphism induced by $\varphi$. Then, $\coho{1}{\operatorname{Gal}(L/k)}{\leftidx{_{\overline{\varphi}}}{\mathcal{T}}(L)} \simeq \coho{1}{\operatorname{Gal}(K/k)}{\leftidx{_\varphi}{\mathcal{T}}(K)}$.
\end{prop}

\noindent \textbf{Notations and conventions.} Throughout this paper, we keep with following notations.
\begin{enumerate}
\item[(1)] Let $K/k$ be a finite Galois extension and $G=\text{Gal}(K/k)$. And $\{g_i \, | \, i=1,\cdots,m\}$ is a set of generators of $G$.
\item[(2)] Let $N$ be the lattice of rank $n$ and $\Sigma$ be a fan in $N$. Denote by $X_{\Sigma}$ the split toric variety over $k$ associated to a fan $\Sigma$.
\item[(3)] Let $\text{Aut}_{\Sigma}$ be a subgroup of $\text{GL}(n,\mathbb{Z})$ which preserves a fan $\Sigma$ and $\text{Aut}_{\Sigma}^T$ be an algebraic group of toric automorphisms of $X_{\Sigma}$.
\item[(4)] Let $\varphi: G\rightarrow \text{Aut}_{\Sigma}\subset \text{GL}(n,\mathbb{Z})$ be a group homomorphism and ${}_{\varphi}\mathcal{T}$ be the torus corresponding to $\varphi$.
\end{enumerate}

\section{Galois Cohomology of Algebraic Tori}
Our objective of this paper is identifying the explicit form of $\coho{1}{G}{\text{Aut}_{\Sigma}^T}$ that classifies all isomorphism classes of $K/k$-forms of a quasi-projective toric variety $X_{\Sigma}$. By Theorem \ref{decent of toric}, this nonabelian cohomology set is same as $ \coprod_{\varphi} \coho{1}{G}{{}_{\varphi}\mathcal{T}(K)}/C_{\text{Aut}_{\Sigma}}(\varphi(G))$. So, it is enough to compute each summand $\coho{1}{G}{{}_{\varphi}\mathcal{T}(K)}/C_{\text{Aut}_{\Sigma}}(\varphi(G))$. When we begin this work, the most difficult part is identifying $\coho{1}{G}{{}_{\varphi}\mathcal{T}(K)}$. In this section, we develop techniques to identify $\coho{1}{G}{{}_{\varphi}\mathcal{T}(K)}$.

\subsection{A Galois cohomology of 1-dimensional tori}
Let $\mathcal{U}=\mathbb{Z}$ be an integer ring with a $G$-action. Define a $G$-module $\mathcal{U}\otimes_{\mathbb{Z}} K^{\times}$ whose $G$-action is given by the simultaneous action of $G$ on $\mathcal{U}$ and its action on $K^{\times}$ as a Galois group. Via the identification $\mathcal{U}\otimes K^{\times} \simeq K^{\times}$ that sends $a\otimes k$ to $k^a$, $\mathcal{U}\otimes K^\times$ is considered as a 1-dimensional split torus $K^\times$. For each $g\in G$, we can write $g\cdot 1 = (-1)^i \in \mathcal{U}$ for some $i \in \{0,1\}$. Then, a $G$-action for $\mathcal{U}\otimes K^\times \simeq K^\times$ is described as follows.
\begin{align}\label{1-dim tori}
 g \cdot k = g(k)^{(-1)^i} \,\, \text{where}\,\, g\in G,\, k \in K^\times
\end{align}

Consider a subgroup $I=\{g \in G \, | \, g \cdot 1 = 1 \, \text{in} \,\, \mathcal{U}\,\}$ of $G$ which is a collection of elements of $G$ that acts on $\mathcal{U}\otimes K^\times \simeq K^{\times}$ as a Galois group.  The generalization of this gives us following.

\begin{defn}\label{def G-int} Let $G$ be a finite Galois group.
\begin{enumerate}
\item If $G$ acts on $\mathcal{U}=\mathbb{Z}$, we say $\mathcal{U}$ as a $G$-integer.
\item For a $G$-integer $\mathcal{U}$, $I = \{g \in G \, | \, g \cdot 1 =1 \,\,\text{in}\,\,\mathcal{U}\,\}$ is called a canonical subgroup with respect to $\mathcal{U}$.
\end{enumerate}
\end{defn}

These definitions is useful for stating theorems, especially in Subsection 3.2. Since every 1-dimensional torus ${}_{\varphi}\mathcal{T}(K)$ is isomorphic to $\mathcal{U}\otimes K^{\times}$ for some $G$-integer $\mathcal{U}$, Theorem \ref{1-dim} gives an answer for a Galois cohomology of 1-dimensional tori.

\begin{thm}\label{1-dim}
For a $G$-integers $\mathcal{U}$, we have $\coho{1}{G}{\mathcal{U} \, \otimes \, K^\times} \simeq \operatorname{Br}(k|K^I)$ where $I$ is the canonical subgroup with respect to $\mathcal{U}$.
\end{thm}
\begin{proof}
We treat $\mathcal{U}\otimes K^\times$ as a 1-dimensional split torus with the $G$-action defined as (\ref{1-dim tori}). When $I=G$, $G$ acts on $\mathcal{U}\otimes K^\times$ as a Galois group. By Hilbert's theorem 90, $\coho{1}{G}{\mathcal{U}\otimes K^\times}$ is trivial and this is same as $\brau{k}{K^G}\simeq1$. Now, consider when $I$ is a proper subgroup of $G$.
From the exact sequence (\ref{les group}), we have
\[1 \longrightarrow \coho{1}{G/I}{\mathcal{U}\otimes K^{\times I}} \overset{\text{inf}} \longrightarrow\coho{1}{G}{\mathcal{U}\otimes K^\times} \overset{\text{res}} \longrightarrow \coho{1}{I}{\mathcal{U}\otimes K^\times}\] 

Since $I$ acts on $\mathcal{U}\otimes K^\times$ as a Galois group,  $\coho{1}{I}{\mathcal{U}\otimes K^\times}$ is trivial by Hilbert's Theorem 90. By above sequence, we have $\coho{1}{G/I}{\mathcal{U}\otimes K^{\times I}} \simeq \coho{1}{G}{\mathcal{U}\otimes K^\times}$. Let $\phi:G \rightarrow \mathbb{Z}_2\simeq \{1,-1\}$ be a group homomorphism defined by $\phi(g) = g \cdot 1$ where $\cdot$ denotes a $G$-action on $\mathcal{U}$. It is easy to check that $\phi$ is surjective map whose kernel is equal to $I$. By applying (\ref{alternating}), we get follows.
\[\coho{1}{G}{\mathcal{U}\otimes K^\times} \simeq \coho{1}{\mathbb{Z}_2}{\mathcal{U}\otimes K^{\times I}} \simeq \coho{2}{\mathbb{Z}_2}{K^{\times I}} \simeq \brau{k}{K^I}
\]
Note that $K^I$ is a field extension of $k$ with degree $\leq 2$. For $a \in k^\times$,  $\overline{a}$ denotes the representation of an element of $\brau{k}{K^I}=k^\times/\text{Im}\,\norm{K^I}{k}$. Each $\overline{a}^{-1} \in \brau{k}{K^I}$ corresponds to a 1-cocycle $c \in \coho{1}{G}{\mathcal{U}\otimes K^\times}$ such that $c_{g} = a$ for $g\notin I$ and $c_{g}=1$ for $g\in I$.
\end{proof}

Now, consider the case ${}_{\varphi}\mathcal{T}(K)$ has dimension $\geq 2$, say its dimension is $n$. By Proposition \ref{computation reduce}, the computation of $\coho{1}{G}{{}_{\varphi}\mathcal{T}(K)}$ is reduced to handle when $\varphi$ is injective. Throughout Subsection 3.2 and 3.3, we assume $\varphi$ is injective and $G$ is a subgroup of $\text{GL}(n,\mathbb{Z})$. In other words, every $g \in G$ is also considered as a matrix $\varphi(g) \in \text{GL}(n,\mathbb{Z})$. Now, our problem can be reduced as follows.

\MyQuote{ How can we compute $\coho{1}{G}{{}_{\varphi}\mathcal{T}(K)}$ and find its generators with the assumption that $G$ is a subgroup of $\text{GL}(n,\mathbb{Z})$?
}\label{problem}

There are three main results : Theorems \ref{induced 1}, \ref{induced 2} and \ref{main thm of 3.3}. These transform the problem (\ref{problem}) into the problem of checking similarlity of matrices over $\mathbb{Z}$. Let me introduce one more definition.

\begin{defn}
Let $A = \{A_i \, | \, 1\leq i \leq t\}$ and $B= \{B_i \, | \, 1\leq i \leq t\}$ be two sets of $n\times n$ matrices over $\mathbb{Z}$. We say $A$ is simultaneously similar to $B$ if there exist $T \in \text{GL}(n,\mathbb{Z})$ such that $T^{-1} A_i T = B_i$ for all $1\leq i \leq t$. 
\end{defn}

\subsection{Induced module and Split torus}
For a subgroup $N$ of $G$ and an $N$-integer $\mathcal{U}$, we define $\ind{N}{G}{\mathcal{U} \otimes K^\times} = \mathbb{Z}G \otimes_{\mathbb{Z}N} \mathcal{U} \otimes_{\mathbb{Z}} K^\times $ with a $G$-action as follows.
\[g \cdot (x\otimes y\otimes z) = gx \otimes y \otimes z, \,\,\, g, x \in G, \,\, y \in \mathcal{U}, \,\, z \in K^\times
\]

\noindent Note that $\coho{1}{G}{\ind{N}{G}{\mathcal{U}\otimes K^\times}}$ can be computed completely. More concretely, it reduced to $\coho{1}{N}{\mathcal{U}\otimes K^\times}$ by Shapiro's Lemma and this is obtained by Theorem \ref{1-dim}. So, we can ask `Can we write $\coho{1}{G}{{}_{\varphi}\mathcal{T}(K)}$ in terms of $\coho{1}{G}{\ind{N}{G}{\mathcal{U}\otimes K^\times}}$?'. This question is reduced to reveal the relation  between ${}_{\varphi}\mathcal{T}(K)$ and $\text{Ind}_N^G(\mathcal{U}\otimes K^\times)$. Before starting this work, we need follows.

\begin{defn}\label{def:operator}
Let $N$ be a subgroup of $G$ and $\mathcal{U}$ be an $N$-integer. For each $g \in G$, we define a linear operator as follows.
\begin{align*}
T_g(N,\mathcal{U}) : \operatorname{Ind}_{N}^{G}(\mathcal{U}) \rightarrow \operatorname{Ind}_{N}^{G}(\mathcal{U}), \quad x\otimes y \rightarrow gx \otimes y 
\end{align*}
\end{defn}

When $N$ and $\mathcal{U}$ can be known by the context, we just write $T_g(N,\mathcal{U})$ as $T_g$. Let $t$ be the index of $N$, $A$ be an element of $\text{GL}(t,\mathbb{Z})$ and $S$ be an ordered $\mathbb{Z}$-basis for $\ind{N}{G}{\mathcal{U}}$. By considering $T_g$ as its matrix representation with respect to $S$, it is possible to check whether $T_g$ is similar to $A$. Since it does not depends on the choice of $S$, this similarity is well-defined. Our first question is when ${}_{\varphi}\mathcal{T}(K)$ is isomorphic to $\ind{N}{G}{\mathcal{U}\otimes K^\times}$ as a $G$-module.

\begin{thm}\label{induced 1}
For a subgroup $N$ of $G$ and an $N$-integer $\mathcal{U}$, suppose that $\{g_i\,|\,i=1,\cdots,m\}$ is simultaneously similar to $\{T_{g_i}(N,\mathcal{U})\,|\,i=1,\cdots,m\}$. Then,
\begin{enumerate}
\item  ${}_{\varphi}\mathcal{T}(K) \simeq \operatorname{Ind}_{N}^G (\mathcal{U}\otimes K^\times)$ as $G$-modules.
\item $\coho{1}{G}{{}_{\varphi}\mathcal{T}(K)} \simeq \operatorname{Br}(K^N|K^I)$ where $I$ is a canonical subgroup with respect to $\mathcal{U}$.
\end{enumerate}
\end{thm}
\begin{proof} 
Define a $G$-module $\ind{N}{G}{\mathcal{U}}\otimes K^\times$ with a $G$-action as follows.
\[g \cdot (x\otimes y\otimes z) = gx \otimes y \otimes g(z), \,\,\, g, x \in G, \,\, y \in \mathcal{U}, \,\, z \in K^\times
\]
Let $G/N=\{t_iN \, | \, i=1,\cdots,n \}$ for some $t_i \in G$. Then, a set $\{e_i = t_i \otimes 1 \,|\, i=1,\cdots,m\}$ is a $\mathbb{Z}$-basis for $\ind{N}{G}{\mathcal{U}}$. For each $k \in K^\times$, consider the map that sends $e_i \otimes k$ to $e_i \otimes t_i(k)$ from $\ind{N}{G}{\mathcal{U}\otimes K^\times}$ to $\ind{N}{G}{\mathcal{U}}\otimes K^\times$. It is easy to check this is a $G$-module isomorphism. \newline

\noindent(1) By the similarity condition, there exists a $\mathbb{Z}$-basis for $\ind{N}{G}{\mathcal{U}}$, denoted by $S=\{f_i\,|\,i=1,\cdots,m\}$, such that a matrix representation of $T_{g_i}$ with respect to $S$ is equal to $g_i \in \text{GL}(n,\mathbb{Z})$. In other words, we have $T_{g}(f_i) = g \cdot \, f_i = \sum_{j=1}^n a_{ji} f_j$ where $g=(a_{ij}) \in G\leq\text{GL}(n,\mathbb{Z})$ and $\cdot$ denotes the $G$-action defined as $(\ref{G-action})$. Now, we shall prove $\text{Ind}_N^G(\mathcal{U})\otimes K^{\times} \simeq {}_{\varphi}\mathcal{T}(K)$ as a $G$-module. For each $k_i \in K^\times$, define a group homomorphism $\psi: \text{Ind}_N^G(\mathcal{U})\otimes K^\times \rightarrow {}_{\varphi}\mathcal{T}(K)$ by $\psi(\sum_{i=1}^n f_i \otimes k_i) = (k_1,k_2,\cdots,k_n)$. Then,
\begin{align*}
\psi(g \cdot (\sum_{i=1}^n f_i \otimes k_i)) &= \psi(\sum_{i=1}^n (a_{1i}f_1+\cdots+a_{ni}f_n) \otimes g(k_i) ) = \psi(\sum_{i=1}^n f_i \otimes g(k_1)^{a_{i1}}g(k_2)^{a_{i2}}\cdots g(k_n)^{a_{in}}) \\
&= ( \prod_{1\leq i \leq n} g(k_i)^{a_{1i}},\cdots, \prod_{1 \leq i \leq n} g(k_i)^{a_{ni}} ) = g \cdot (k_1,\cdots,k_n)
\end{align*}
Hence, $\psi$ preserves the $G$-action. Since $\psi$ is bijective, it is a $G$-module isomorphism.
\newline \newline
\noindent (2) Now, we compute $\coho{1}{G}{{}_{\varphi}\mathcal{T}(K)}$ using Shaprio's Lemma as follows.
\begin{align}\label{isom of induced1}
\coho{1}{G}{{}_{\varphi}\mathcal{T}(K)} \simeq \coho{1}{G}{\text{Ind}_N^G(\mathcal{U})\otimes K^{\times}} \simeq \coho{1}{G}{\text{Ind}_N^G(\mathcal{U} \otimes K^{\times})} \simeq \coho{1}{N}{\mathcal{U}\otimes K^{\times}}
\end{align}
If we consider $N$ as $\text{Gal}(K/K^N)$, we get $\coho{1}{N}{\mathcal{U}\otimes K^{\times}} 
\simeq \brau{K^N}{K^{I}}$ by Theorem \ref{1-dim}
\end{proof}

\begin{exa}\label{induced 1 exa}
Suppose that $G\simeq D_6 \leq \operatorname{GL}(3,\mathbb{Z})$ and generated by following matrices.
\[\footnotesize{
\left\{r=\begin{pmatrix} 
0 & 1 & 0  \\ 
0 & 0 & 1 \\
1 & 0 & 0
\end{pmatrix},\,
s=\begin{pmatrix}
0 & 0 & -1 \\
0 & -1 & 0 \\
-1 & 0 & 0
\end{pmatrix} \right\}}
\]
Note that $r$ is a rotation and $s$ is a reflection of $D_6$. Take a subgroup $N=\{1,s\}$ and an $N$-integer $\mathcal{U}$ whose $N$-action is defined by $s \cdot 1 = -1$. Since $G/N=\{N,rN,r^2N\}$, the set $\{e_i=r^{i-1}\otimes 1 \, |\, i=1,2,3 \}$ is a $\mathbb{Z}$-basis for $\operatorname{Ind}_{N}^G(\mathcal{U})$. Then,
\begin{align*}
T_s(e_1) = s \otimes 1 = 1 \otimes s\cdot 1 = -e_1, \,\, T_s(e_2) = sr \otimes 1 = r^2 \otimes s\cdot 1 = -e_3,\, 
T_s(e_3) = - e_2
\end{align*}
With a similar calculation of $T_r$, matrix representations of $T_r$ and $T_s$ with respect to $\{e_i\,|\,i=1,2,3\}$ are obtained as follows.
\[\footnotesize{
T_r = \begin{pmatrix} 0&0&1 \\ 1&0&0 \\ 0&1&0 \end{pmatrix},\,\, T_s = \begin{pmatrix} -1&0&0 \\ 0&0&-1 \\ 0&-1&0 \end{pmatrix}, \,\, T= \begin{pmatrix} 0&1&0 \\ 1&0&0 \\ 0&0&1 \end{pmatrix}}
\] 
For a matrix $T\in \text{GL}(3,\mathbb{Z})$ defined as above, we have $T^{-1}rT=T_r$ and $T^{-1}sT = T_s$. By Theorem \ref{induced 1}, we have $\coho{1}{G}{{}_{\varphi}\mathcal{T}(K)} \simeq \operatorname{Br}(K^{\left<s\right>}|K)$.
\end{exa}

Although we have identified the explicit form of $\coho{1}{G}{{}_{\varphi}\mathcal{T}(K)}$, we don't know anything about its generators. Due to the complicacy of the isomorphism of Shapiro's Lemma, a deep analysis of this map is needed.

Let $N$ be the subgroup in Theorem \ref{induced 1} and $G/N = \{t_iN\,| \, i=1,\cdots,n\}$. If we consider each $T_{g_i}$ as its matrix representation with respect to $\{e_i = t_i\otimes 1 \, | \, i=1,\cdots,n\}$, the similarity condition can be written as $T^{-1}g_iT = T_{g_i}$ for some $T\in \text{GL}(n,\mathbb{Z})$. It is easy to see that the set of right coset of $N$ is same as $\{Nt_{i}^{-1} \, | \, i=1,\cdots,n\}$. For each $i\in \{1,\cdots,n\}$, there exists an unique $j\in \{1,\cdots,n\}$ such that $t_i^{-1}g \in Nt_{j}^{-1}$,  equivalently, $t_i^{-1}gt_j \in N$. For each $g\in G$ and $i \in \{1,\cdots,n\}$, we define an integer $b_i(g)$ as follows.
\begin{align}\label{integer}
b_i(g) = \begin{cases}
1 \quad \text{when}\,\,\, t_i^{-1}gt_j \notin I \\
0 \quad \text{when}\,\,\, t_i^{-1}gt_j \in I
\end{cases}
\end{align}

\begin{cor}\label{generator form}
Suppose that Theorem \ref{induced 1} holds. Then, $\coho{1}{G}{{}_{\varphi}\mathcal{T}(K)}$ is generated by $\{T\, *\, c(a) \, | \,a\in K^\times, \, g(a)=a \,\,\text{for all}\,\, g\in N \}$, where $*$ is an action as (\ref{matrix action}) and $c(a):G\rightarrow {}_{\varphi}\mathcal{T}(K)$ is defined as $c(a)_g = (t_1(a)^{b_1(g)},t_2(a)^{b_2(g)},\cdots,t_n(a)^{b_n(g)})$.
\end{cor}

\begin{proof}
Remind the isomorphism $\coho{1}{G}{\ind{N}{G}{\mathcal{U}} \otimes K^\times} \simeq \coho{1}{N}{\mathcal{U}\otimes K^{\times}} \simeq \brau{K^{N}}{K^{I}}$ of (\ref{isom of induced1}).
For each $a \in K^\times$, $\overline{a}$ denotes the representation of an element of $\brau{K^N}{K^I}$. In this proof, we treat $\mathcal{U}\otimes K^\times$ as a 1-dimensional split torus $K^\times$ whose $G$-action is given by (\ref{1-dim tori}). By Theorem \ref{1-dim}, we see that $\overline{a}^{-1} \in \brau{K^{N}}{K^I}$ corresponds to a 1-cocycle $d \in \coho{1}{N}{\mathcal{U}\otimes K^\times}$ such that $d_n=a$ when $n \notin I$ and $d_n=1$ when $n \in I$. The difficult part is identifying an element of $\coho{1}{G}{\ind{N}{G}{\mathcal{U}} \otimes K^{\times}}$ that corresponds to a 1-cocycle $d$.

Let $G_i = \mathbb{Z}G \otimes_{\mathbb{Z}} \cdots \otimes_{\mathbb{Z}} \mathbb{Z}G$ and $N_i = \mathbb{Z}N \otimes_{\mathbb{Z}} \cdots \otimes_{\mathbb{Z}} \mathbb{Z}N$, in both case $\otimes$ is written $(i+1)$-times. Consider the two bar resolutions of $\mathbb{Z}$ as follows.
\[
\xymatrix{\cdots \ar[r]^{d_1} & G_1 \ar[r]^{d_0}\ar[d]^{i_1} & G_0 \ar[r]^{\text{aug}}\ar[d]^{i_0} &\mathbb{Z} \ar[r]\ar[d]^{=} &0 \\
\cdots \ar[r]^{d_1}& N_1 \ar[r]^{d_0}& N_0 \ar[r]^{\text{aug}} &\mathbb{Z} \ar[r] & 0 }
\]
Since every element of $G$ is written as $nt_{i}^{-1}$ for some $n \in N$ and $i \in\{1,\cdots,m\}$, we define a set map $i_0:G_0 \rightarrow N_0$ by $i_0(nt_{i}^{-1})=n$. By considering each $G_0$ and $N_0$ as a free $N$-module, $i_0$ can be extended to an $N$-module morphism. Now, fix an element $g \in G$ and $i \in \{1,\cdots,n\}$. We can write $t_i^{-1}g = n_i t_j^{-1}$ for some $n_i \in N$ and $j \in \{1,\cdots,n\}$. Define a set map by $i_1(t_i^{-1}\otimes g) = 1\otimes n_i$ and this can be extended to an $N$-module map which makes above diagram commute. By taking $\text{Hom}_{\mathbb{Z}N}(\cdot,\mathcal{U}\otimes K^\times)$, we have follows.
\[
\xymatrix{0 \ar[r] & \text{Hom}_{\mathbb{Z}N}(N_0, \mathcal{U}\otimes K^\times) \ar[r]^{d_0}\ar[d]^{i_0}& \text{Hom}_{\mathbb{Z}N}(N_1, \mathcal{U}\otimes K^\times) \ar[d]^{i_1}\ar[r]^{\quad\quad\quad\quad d_1} & \cdots \\
0 \ar[r] & \text{Hom}_{\mathbb{Z}N}(G_0, \mathcal{U}\otimes K^\times) \ar[r]^{d_0} & \text{Hom}_{\mathbb{Z}N}(G_1,\mathcal{U}\otimes K^\times) \ar[r]^{\quad\quad\quad\quad d_1} & \cdots
}
\]

Consider $\coho{1}{N}{\mathcal{U}\otimes K^\times}$ as the first cohomology from the above chain. Then, an element $f \in \text{Hom}_{\mathbb{Z}N}(N_1,\mathcal{U}\otimes K^{\times})$, defined by $f(1\otimes n) = d_n$ for all $n \in N$, is a representation of a 1-cocycle $d$. Let $F =f \circ i_1 \in \text{Hom}_{\mathbb{Z}N}(G_1, \mathcal{U}\otimes K^\times)$ and we have $F(t_{i}^{-1} \otimes g) = f(1 \otimes n_i)=d_{n_i}=a^{b_i(g)}$. Consider the isomorphism $\text{Hom}_{\mathbb{Z}N}(G_1, \mathcal{U} \otimes K^\times) \simeq \text{Hom}_{\mathbb{Z}G}(G_1,\text{Ind}_N^G(\mathcal{U})\otimes K^\times)$ that sends $F$ to $G$, where $G$ is defined by
\[G(1\otimes g) = \sum_{i=1}^{n} t_{i} \otimes t_i\big(F(t_{i}^{-1} \otimes g)\big) = \sum_{i=1}^{n} t_i \otimes t_i(a)^{b_i(g)}\]
By sending $\sum_{i=1}^n t_i \otimes k_i$ to $(k_1,\cdots,k_n)$, we can identify $\ind{N}{G}{\mathcal{U}}\otimes K^\times$ with ${}_{\varphi}\mathcal{T}(K)$. Then, $G$ corresponds to a 1-cocycle $c(a):G \rightarrow {}_{\varphi}\mathcal{T}(K)$ such that $c(a)_g = (t_1(a)^{b_1(g)},\cdots,t_n(a)^{b_n(g)})$. And $c(a)$ corresponds to $d$ via $\coho{1}{N}{\mathcal{U}\otimes K^{\times}} \simeq \coho{1}{G}{\ind{N}{G}{\mathcal{U}} \otimes K^\times}$. 

Let $\mathcal{B}=\{f_i\,|\,i=1,\cdots,t\}$ be a $\mathbb{Z}$-basis for $\mathbb{Z}^n$ such that a matrix representation of $T_{g_i}$ with respect to $\mathcal{B}$ is same as a matrix $g_i$. By considering the value of $c(a)$ as an element of $K^n\simeq\mathbb{Z}^n\otimes K$, we need to change the coordinate of $c(a)$ with respect to $K$-basis $\{e_i=f_i\otimes 1 \, | \, i=1,\cdots,n\}$. This work can be achived by multiplying $T$ on $\mathbb{Z}^n$, so $\{T *\, c(a)\,|\, g(a)=a\text{ for all}\,\, g\in N\}$ is a set of generators of $\coho{1}{G}{{}_{\varphi}\mathcal{T}(K)}$.
\end{proof}

\begin{exa}\label{generator exa}
In this example, we find generators of $\coho{1}{G}{{}_{\varphi}\mathcal{T}(K)}$ in Example \ref{induced 1 exa}. Note that $N=\{1,s\}$, $I$ is a trivial group and $G/N=\{N,rN,r^2N\}$. Use Corollary \ref{generator form} in the setting of $a_1=1$, $a_2=r$ and $a_3=r^2$. It is easy to check that $a_{1}^{-1}s a_1$, $a_{2}^{-1}sa_3$ and $a_{3}^{-1}sa_2$ are in $N$ but not in $I$. So, $b_1(s)=b_2(s)=b_3(s)=1$. Similarly, we get $b_1(r)=b_2(r)=b_3(r)=0$. For a nonzero $a \in K^N$, the 1-cocycle $c(a)$ of Corollary \ref{generator form} is defined by $c(a)_r=(1,1,1)$ and $c(a)_s=(a,a,a)$. Applying $T$-action as (\ref{matrix action}) on the 1-cocycle $c(a)$ induces the same cocycle. So, $\{c(a)\,|\, a\in K^\times \,\, \text{with}\,\, s(a)=a \}$ generates $\coho{1}{G}{{}_{\varphi}\mathcal{T}(K)}$.
\end{exa}

See Table 3,4 of Appendix A. When $n=2$ or $n=3$, we summarized a conjugacy class of $G$ in $\text{GL}(n,\mathbb{Z})$, a subgorup $N$ of $G$ and an $N$-integer $\mathcal{U}$ that satisfy Theorem \ref{induced 1}.  For these cases, $\coho{1}{G}{{}_{\varphi}\mathcal{T}(K)}$ can be computed as Example \ref{induced 1 exa} and \ref{generator exa}. Our next question is when ${}_{\varphi}\mathcal{T}(K)$ is a direct summand of $\ind{N}{G}{\mathcal{U}\otimes K^\times}$ as a $G$-module.

\begin{thm}\label{induced 2}
Let $\mathcal{U}$ be a $G$-integer, $N$ be a subgroup of $G$ of index $n+1$ and $I$ be a canonical subgroup with respect to $\mathcal{U}$. For each $g_i \in G$, we can write $g_i \cdot 1 = (-1)^{a_i}\in \mathcal{U}$ for some $a_i \in \{0,1\}$. Suppose that $\{T_{g_i}(N,\mathcal{U})\,|\,i=1,\cdots,m\}$ is simultaneously similar to follows. In lower right side of each $u_i\in \operatorname{GL}(n+1,\mathbb{Z})$, each $g_i\in \operatorname{GL}(n,\mathbb{Z})$ is located as a submatrix. Then,
\[\footnotesize{
\left\{\, u_1 = \begin{pmatrix}
(-1)^{a_1} & 0 & \cdots & 0 \\
* & \\
\vdots & & \Huge{g_1} \\
* & 
\end{pmatrix},
\cdots,
u_m = \begin{pmatrix}
(-1)^{a_m} & 0 & \cdots & 0 \\
* & \\
\vdots & & \Huge{g_m} \\
* & 
\end{pmatrix} \,
\right\}}
\]
\begin{enumerate}
\item There exists a $G$-module $M$ such that $\operatorname{Ind}_{N}^G(\,\mathcal{U}\otimes K^\times) \simeq M\oplus {}_{\varphi}\mathcal{T}(K)$ as a $G$-module.

\item When $N$ is a subgroup of $I$, $\coho{1}{G}{{}_{\varphi}\mathcal{T}(K)}$ is generated by a set of 1-cocycles $\{c(a) \, | \, a\in K^{\times} \,\,\text{with}$ $g_i (a)=a^{(-1)^{a_i}}\, \text{for all}\,\,i\,\}$. Moreover, $c(a)$ is defined as $c(a)_{g_i} = (a^{b_{1i}(-1)^{a_i}},a^{b_{2i}(-1)^{a_i}}$\\$, a^{b_{ni}(-1)^{a_i}} )$ where $b_{ji}$ as an $(j+1,1)$-th entry of matrix $r_i$. 
\end{enumerate}  
\end{thm}
\begin{proof}
By the similarity condition, there exists a $\mathbb{Z}$-basis $S=\{v_1,\cdots,v_{n+1}\}$ for $\ind{N}{G}{\mathcal{U}}$ such that a matrix representation of $T_{g_i}$ with respect to $S$ is same as $u_i$. \newline

\noindent(1) Define a surjective $\mathbb{Z}$-linear map $\pi: \ind{N}{G}{\mathcal{U}} \rightarrow \mathcal{U}$ by $\pi(v_1)=1$ and $\pi(v_i)=0$ for $2\leq i \leq n+1$. Since $g_i \cdot v_1 = (-1)^{a_i}v_1 + \sum_{j=1}^{n}b_{ji}v_{j+1}$, we have $\pi(g_i \cdot v_1) = (-1)^{a_i} = g_i \cdot \pi(v_1)$ for all $i$. This implies $\pi$ is actually a $G$-module homomorphism. Now, consider a $G$-module $\ker{\pi}\otimes K^\times$ whose $G$-action is given by the simultaneous action of $G$ on $\ker{\pi}$ and its action on $K^\times$ as a Galois group. Since $\{v_2,\cdots,v_{n+1}\}$ is a $\mathbb{Z}$-basis for $\ker{\pi}$, $T_{g_i} |_{\ker{\pi}}$ is same as $g_i$. By Theorem \ref{induced 1}-(1), $\ker{\pi}\otimes K^{\times}$ is isomorphic to ${}_{\varphi}\mathcal{T}(K)$ as a $G$-module. Since $\mathcal{U}$ is a free $\mathbb{Z}$-module, the natural exact sequence $0 \rightarrow \ker{\pi} \rightarrow \ind{N}{G}{\mathcal{U}} \rightarrow \mathcal{U} \rightarrow 0$ splits. By tensoring $K^\times$, we check that $\ker{\pi}\otimes K^\times \simeq {}_{\varphi}\mathcal{T}(K)$ is a direct summand of $\ind{N}{G}{\mathcal{U}}\otimes K^\times$. \newline

\noindent(2) Let $\pi: G \rightarrow \text{GL}(n+1,\mathbb{Z})$ be a group homomorphism that sends $g_i$ to $u_i$. We obtain an exact sequence of $G$-modules with $i(k_1,\cdots,k_n) = (1,k_1,\cdots,k_{n})$ and $p(k_1,\cdots,k_{n+1}) = k_1$.
\begin{align}\label{first seq}
&1 \longrightarrow {}_{\varphi}\mathcal{T}(K) \overset{i}\longrightarrow {}_{\pi}\mathcal{T}(K) \overset{p}\longrightarrow  \mathcal{U} \otimes K^{\times}\simeq K^\times \longrightarrow 1
\end{align}

\noindent Also, (\ref{les module}) gives the following long exact sequence.
\begin{align*}
\coho{0}{G}{{}_{\pi}\mathcal{T}(K)} \overset{p^0}\rightarrow &\coho{0}{G}{\mathcal{U}\otimes K^{\times}} \overset{\delta}\rightarrow \coho{1}{G}{{}_{\varphi}\mathcal{T}(K)}\overset{i^1}\rightarrow \coho{1}{G}{{}_{\pi}\mathcal{T}(K)} \overset{p^1}\rightarrow \coho{1}{G}{\mathcal{U}\otimes K^{\times}} \rightarrow \cdots
\end{align*}

\noindent From this, we extract a short exact sequence as follows.
\begin{align}\label{second seq}
1 \longrightarrow \text{im}\,{\delta} \longrightarrow \coho{1}{G}{{}_{\varphi}\mathcal{T}(K)} \longrightarrow \ker{p^1} \longrightarrow 1
\end{align}

By Theorem \ref{induced 1}-(1), we have $\ind{N}{G}{\mathcal{U}}\otimes K^\times \simeq {}_{\pi}\mathcal{T}(K)$ as $G$-modules. When $N$ is a subgroup of $I$, we get $\coho{1}{G}{{}_{\pi}\mathcal{T}(K)} \simeq \coho{1}{G}{\ind{N}{G}{\mathcal{U}} \otimes K^{\times}} \simeq \coho{1}{N}{\mathcal{U} \otimes K^\times} \simeq 1$ by Hilbert's Theorem 90. Hence, $\ker{\pi^1}$ is clearly trivial and $\coho{1}{G}{{}_{\varphi}\mathcal{T}(K)} \simeq \text{im}\,\delta$. 

Now, we verify $\text{im}\,\delta$ using the definition of $\delta$, see the paragraph after (\ref{les module}). It is easy to see that ${}_{\pi}\mathcal{T}(K)/{}_{\varphi}\mathcal{T}(K) = \{(a,1,\cdots,1) \, | \, a\in K^\times\}$ where $(a,1,\cdots,1)$ denotes a representation of an element of the left coset. Choose any $(a,1,\cdots,1) \in ({}_{\pi}\mathcal{T}(K)/{}_{\varphi}\mathcal{T}(K))^{\,G}$ and observe follows.
\begin{align*}
g_i \cdot (a,1,\cdots,1) & = (g_i(a)^{(-1)^{a_i}},*,\cdots,*) = (g_i(a)^{(-1)^{a_i}},1,\cdots,1)
\end{align*}
From the relation $g_i \cdot (a,1,\cdots,1) = (a,1,\cdots,1)$, we get $g_i(a)^{(-1)^{a_i}} = a$ for all $i$. For such $a \in K^\times$, let $c(a)$ be a 1-cocyle which is a representation of $\delta(a) \in \coho{1}{G}{{}_{\varphi}\mathcal{T}(K)}$. By the definition of $\delta$, we have
\begin{align*}
c(a)_{g_i} &= (a,1,\cdots,1)^{-1} \,\,(g_i\cdot (a,1,\cdots,1)) = (a^{-1},1,\cdots,1)\,\, (g_i(a)^{(-1)^{a_i}}, g_i(a)^{b_{1i}},\cdots,g_i(a)^{b_{ni}}) \\
&= (1,a^{b_{1i}(-1)^{a_i}},\cdots,a^{b_{ni}(-1)^{a_i}})
\end{align*}
Clearly, above elements corresponds to $(a^{b_{1i}(-1)^{a_i}},a^{b_{2i}(-1)^{a_i}}, \cdots, a^{b_{ni}(-1)^{a_i}}) \in {}_{\varphi}\mathcal{T}(K)$.
\end{proof}

\begin{exa}\label{induced 2 exa}
Suppose that $G\simeq \mathbb{Z}_2 \times \mathbb{Z}_2 \leq \operatorname{GL}(3,\mathbb{Z})$ and generated by following matrices.
\[\footnotesize{\left\{r= \begin{pmatrix}
-1 & 0 & 0 \\
0 & 0 & 1 \\
0 & 1 & 0
\end{pmatrix}, s = \begin{pmatrix}
-1 & 0 & 0 \\
1 & 0 & -1 \\
-1 & -1 & 0
\end{pmatrix}
\right\}}
\]
Let $N$ be a trivial subgrup of $G$ and $\mathcal{U}$ be a $G$-integer whose $G$-action is defined by $r \cdot 1 = s \cdot 1 = 1$. The set $\{ e_1 = 1\otimes 1, e_2 = r \otimes 1, e_3 = s\otimes 1, e_4 = rs \otimes 1 \}$ is a $\mathbb{Z}$-basis for $\operatorname{Ind}_N^G(\mathcal{U})$. Then,
\begin{align*}
&T_r(e_1) = e_2,\, T_r(e_2) = e_1,\, T_r(e_3) = e_4,\, T_r(e_4) = e_3,\\
&T_s(e_1) = e_3, \, T_s(e_2) = e_4, \, T_s(e_3) = e_1, \, T_s(e_4) = e_2
\end{align*}
Hence, each matrix representations of $T_r$ and $T_s$ are given as follows.
\[\footnotesize{T_r =\begin{pmatrix}
0 & 1 & 0 & 0 \\
1 & 0 & 0 & 0 \\
0 & 0 & 0 & 1 \\
0 & 0 & 1 & 0
\end{pmatrix}, \,\, T_s = \begin{pmatrix}
0 & 0 & 1 & 0 \\
0 & 0 & 0 & 1 \\
1 & 0 & 0 & 0 \\
0 & 1 & 0 & 0
\end{pmatrix}, \,\, T = \begin{pmatrix}
1 & -1 & 0 & 0 \\
0 & 1 & 1 & 0 \\
0 & -1 & -1 & -1 \\
0 & 1 & 0 & 1
\end{pmatrix}}
\]
By observing $T^{-1}T_rT$ and $T^{-1}T_sT$, we check the similarity condition of Theorem \ref{induced 2} .
\[
\footnotesize{T^{-1}T_rT = \begin{pmatrix}
1 & 0 & 0 & 0 \\
1 & -1 & 0 & 0 \\
0 & 0 & 0 & 1 \\
0 & 0 & 1 & 0
\end{pmatrix}, \,\, T^{-1}T_sT = \begin{pmatrix}
1 & 0 & 0 & 0 \\
1 & -1 & 0 & 0 \\
0 & 1 & 0 & -1 \\
1 & -1 & -1 & 0
\end{pmatrix}}\]
It is clear that $N$ is contained in the canonical subgroup with respect to $\mathcal{U}$. By Theorem \ref{induced 2}-(2), $\coho{1}{G}{{}_{\varphi}\mathcal{T}(K)}$ is generated by $\{c(a) \, | \, a\in k^\times\}$ where $c(a)_r = (a,1,1), \,\, c(a)_s = (a,1,a)$. A calculaton shows that $c(a)$ is cohomologous to $c(a^{\prime})$ iff $a^{\prime} =a \norm{K^{\left<sr\right>}}{k}{(b)}$ for some $b \in (K^{\left<sr\right>})^{\times}$. Hence, $\coho{1}{G}{{}_{\varphi}\mathcal{T}(K)} \simeq \operatorname{Br}(k|K^{\left<sr\right>})$ via the isomorphism $c(a)\mapsto \overline{a}$ where $\overline{\phantom{a}}$ denotes the representation of an element in Brauer group.
\end{exa}

The key idea of Theorem \ref{induced 2} is constructing (\ref{second seq}) and reveals the structure of $\text{im}\,\delta$ when $\coho{1}{G}{{}_{\phi}\mathcal{T}(K)}$ is trivial. Unfortunately, $\coho{1}{G}{{}_{\phi}\mathcal{T}(K)}$ is not trivial when $N$ is not a subgroup of $I$, so Theorem \ref{induced 2} does not work. By adding more conditions of $G$, we can write $\coho{1}{G}{{}_{\varphi}\mathcal{T}(K)}/\text{im}\,\delta$ in terms of Brauer groups.

\begin{cor}\label{cor of induced 2}
We follow the notations in Theorem \ref{induced 2}. Let $\{r,s\}$ be the set of generators of $G$ and $M$ be a subgroup of $\coho{1}{G}{{}_{\varphi}\mathcal{T}(K)}$ generated by 1-cocycles in Theorem \ref{induced 2}-(2). Suppose that $G$ is a dihedral group or an abelian group, and $G/N= \{r^iN \, | \, i=0,\cdots,n \}$. When $G$ is a dihedral group, $r$ denotes a rotation and $s$ denotes a reflection.  Then, $\coho{1}{G}{{}_{\varphi}\mathcal{T}(K)}/M \simeq \ker{\eta}$ where $\eta$ is a group homomorphism $\eta: \operatorname{Br}(K^N |K^{I \cap N}) \rightarrow \operatorname{Br}(k|K^I)$ induced by a norm map $\norm{K^N}{k}:K^N\rightarrow k$.
\end{cor}
\begin{proof}
We follow the notations in the proof of Theorem \ref{induced 2}. From (\ref{second seq}), we have $\coho{1}{G}{{}_{\varphi}\mathcal{T}(K)} / M$ $\simeq \ker{p^1}$. Note that the domain of $p^1$ is $\coho{1}{G}{{}_{\pi}\mathcal{T}(K)} \simeq \coho{1}{G}{\ind{N}{G}{\mathcal{U}}\otimes K^\times} \simeq \brau{K^N}{K^{I\cap N}}$ and codomain of $p^1$ is $\coho{1}{G}{\mathcal{U}\otimes K^\times}\simeq \brau{k}{K^I}$. With this identification, we should verify $p^1$ is same as $\eta$. By the condition of $G/N$, the set of generators of $N$ can be written as $\{r^{n+1},sr^i\}$ for some $i\in \mathbb{Z}$. Without loss of generality, we may assume $i=0$. In this proof, $\overline{\phantom{a}}$ denotes the representative of an element of Brauer group.

To find generators of $\coho{1}{G}{{}_{\pi}\mathcal{T}(K)}$, we should apply Corollary \ref{generator form} in the setting of $a_i=r^{i-1}$. Let $c(a)$ be a 1-cocycle defined in Corollray \ref{generator form} for each $a \in (K^N)^\times$. By Corollary \ref{generator form}, there exist some $T \in \text{GL}(n,\mathbb{Z})$ such that $T *c(a)$ is the generator of $\coho{1}{G}{{}_{\pi}\mathcal{T}(K)}$. If $c(a)_g=(a_1,\cdots,a_{n+1})$ for $g \in G$, then $(p^1(T*c))_g = \prod_{1\leq i \leq n+1} r^{i-1}(a_i)$. 

Now, we handle when $G$ is a dihedral group. Suppose that $r$ and $s$ are not in $I$. Let $b_i(r)$ and $b_i(s)$ be integers defined as (\ref{integer}). First, consider when $n$ is even. From the relation $1^{-1}s1 \notin I$, $r^{-1}sr^{n} \in I, \cdots, r^{-n}sr\in I$, we get $b_1(s)=1$ and $b_i(s)=0$ for all $2\leq i \leq n+1$. Similarly, we get $b_1(r)=1$ and $b_2(i)=0$ for all $2\leq i \leq n+1$. By Corollary \ref{generator form}, we get a 1-cocycle $c(a):G\rightarrow K^{\times n+1}$, where $c(a)_r=(a,1,\cdots,1)$ and $c(a)_s=(a,1,\cdots,1)$ for $a \in (K^N)^\times$. For some $T \in \text{GL}(n,\mathbb{Z})$, $T *c(a)$ is one of the generators of $\coho{1}{G}{{}_{\pi}\mathcal{T}(K)}$ and we denote this cocycle as $d$. We have $(p^1 d)_r =(p^1 d)_s = a$ and $t^{-1}(p^1 d)_r (r\cdot t) = t^{-1}(p^1 d)_s (s\cdot t) = \norm{K^N}{k}{(a)}$ for $t=(r(a)r^3(a)\cdots r^{n-1}(a))^{-1}$. Clearly, this 1-cocycle corresponds to $\overline{\norm{K^N}{k}{(a)}} \in \brau{k}{K^I}$. Hence, we checked $p^1$ is same as $\eta$ in this case. We can apply same analysis for the remaining cases : when $r,s \notin I$ and $n$ is odd, when $r,s\in I$, when $r \notin I$, $s\in I$ and when $r,s\notin I$.

Next, we handle when $G$ is an abelian group. Let me show when $r \in I$ and $s \notin I$. Then, $\overline{a} \in \brau{K^N}{K^I}$ corresponds to a 1-cocycle $c(a)$, where $c(a)_r=(1,\cdots,1)$ and $c(a)_s = (a,r(a),\cdots$ \\ $,r^n(a))$. For some $T \in \text{GL}(n,\mathbb{Z})$, $T *c(a)$ is one of the generators of $\coho{1}{G}{{}_{\pi}\mathcal{T}(K)}$ and we denote this cocycle as $d$. We have $(p^1 d)_r=1$, $(p^1 d)_s = ar(a)\cdots r^n(a) = \norm{K^N}{k}{(a)}$ and this cocycle corresponds to $\overline{\norm{K^N}{k}{(a)}}\in \brau{k}{K^I}$. Same analysis holds on other cases.
\end{proof} 

\begin{exa}\label{induced 2 exa 2}
Suppose that $G \simeq D_6 \leq \operatorname{GL}(2,\mathbb{Z})$ and generated by following matrices.
\[\footnotesize{ \left\{r= \begin{pmatrix}
0&-1 \\
1&-1 
\end{pmatrix},\, s= \begin{pmatrix} 0&1 \\ 1&0 \end{pmatrix}
\right\}}
\]
Note that $r$ is a rotation and $s$ is a reflection. Let $N = \{1,s\}$ and $\mathcal{U}$ be a $G$-integer defined by $r\cdot 1 = 1$, $s \cdot 1 = -1$. Since $G/N = \{N,rN,r^2N\}$, the set $\{e_i =r^{i-1}\otimes 1 \, | \, i=1,2,3\}$ is a $\mathbb{Z}$-basis for $\operatorname{Ind}_{N}^G(\mathcal{U})$. Matrix representations of $T_r$ and $T_s$ with respect to $\{e_1,e_2,e_3\}$ are given as follows.
\begin{align*}
\footnotesize{
T_r = \begin{pmatrix} 0&0&1 \\ 1&0&0 \\ 0&1&0 \end{pmatrix}, \,\, T_s = \begin{pmatrix} -1&0&0 \\ 0&0&-1 \\ 0&-1&0 \end{pmatrix}, \,\, T = \begin{pmatrix} 1&-1&1 \\ 0&0&-1 \\ 0&1&0 \end{pmatrix}}
\end{align*}
By observing $T^{-1}T_rT$ and $T^{-1}T_sT$, we can check the condition of Theorem \ref{induced 2}.
\[\footnotesize{T^{-1}T_rT = \begin{pmatrix} 1&0&0 \\ 0&0&-1 \\ -1&1&-1 \end{pmatrix},\,\, T^{-1}T_sT = \begin{pmatrix} -1&0&0 \\ 0&0&1 \\ 0&1&0 \end{pmatrix}}
\]

\noindent Let $\eta: \operatorname{Br}(K^{\left<s\right>}|K) \rightarrow \operatorname{Br}(k|K^{\left<r\right>})$  be a group homomorphism induced by $\norm{K^{\left<s\right>}}{k}: K^{\left<s\right>} \rightarrow k$. Denote by $\overline{\phantom{a}}$ the representative of an element of Brauer group. Suppose that $\eta(\overline{a}) = \overline{1}$ for some $a \in K^{\left<s\right>}$, equivalently, there exist some $b \in (K^{\left<r\right>})^\times$ such that $ar(a)r^2(a)=bs(b)$. Then, $a = \frac{b}{r(a)}s(\frac{b}{r(a)}) \in \text{Im}\,\norm{K}{K^{\left<s\right>}}$ and so $\ker{\eta}$ is trivial. By Corollary \ref{cor of induced 2}, $\coho{1}{G}{{}_{\varphi}\mathcal{T}(K)} \simeq M$ and $M$ is generated by $\{c(a)\,|\,a\in K^\times \,\text{with}\,\,r(a)=a,\,s(a)=a^{-1}\}$, where  $c(a)_r = (1,a^{-1}),\,c(a)_s=(1,1)$. A calculation shows that $c(a)$ is cohomologous to $c(a^{\prime})$ iff $a^{\prime} = a\norm{K}{K^{\left<r\right>}}{(x)}$ for some $x\in K^\times$ satisfying $\norm{K}{K^{\left<sr\right>}}{(x)}=1$. Hence, the isomorphism that sends $c(a)$ to $\overline{a}$ gives the follows.
\begin{align}\label{M}
\coho{1}{G}{{}_{\varphi}\mathcal{T}(K)} \simeq M \simeq \frac{\{a\in K^\times \, | \, r(a)=a,s(a)=a^{-1}\}}{\{\norm{K}{K^{\left<r\right>}}{(b)}\,|\,b \in K^\times \,\, \text{with} \,\, \norm{K}{K^{\left<sr\right>}}{(b)}=1\}}
\end{align}
\end{exa}

See Table 5,6 of Appendix A. When $n=2$ or $n=3$, we summarized a conjugacy classes of $G$ in $\text{GL}(n,\mathbb{Z})$, a subgroup $N$ of $G$ and a $G$-integer $\mathcal{U}$ that satisfy the similarity condition of Theorem \ref{induced 2}. For these cases, we can compute $\coho{1}{G}{{}_{\varphi}\mathcal{T}(K)}$ as Example \ref{induced 2 exa} and \ref{induced 2 exa 2}.

\subsection{Dimension reduction} 
Suppose that $\{g_i\,|\,i=1,\cdots,m\}$ is simultaneously similar to follows.
\begin{align}\label{block similarity}
\left\{
r_i = \begin{pmatrix}
A_i & O \\ 
C_i & B_i
\end{pmatrix} \, | \, 
\begin{tabular}{c}
$1\leq i \leq m,\, A_i\in \text{GL}(t,\mathbb{Z}),\, B_i \in \text{GL}(s,\mathbb{Z})$ \\
$C_i$ is an $s \times t$ matrix, $O$ is an zero matrix
\end{tabular}
\right\}
\end{align}
There is no condition of natural number $t$ and $s$ except $t+s=n$. This subsection is devoted to solve the problem (\ref{problem}) with this similarity condition. Let $\pi:G \rightarrow \text{GL}(t,\mathbb{Z})$ be a group homomorphism that sends $g_i$ to $A_i$ and $\psi:G\rightarrow \text{GL}(s,\mathbb{Z})$ that sends $g_i$ to $B_i$. It is natural to ask whether $\coho{1}{G}{{}_{\varphi}\mathcal{T}(K)}$ can be written as a direct sum of $\coho{1}{G}{{}_{\pi}\mathcal{T}(K)}$ and $\coho{1}{G}{{}_{\psi}\mathcal{T}(K)}$. This work transforms the problem (\ref{problem}) for an $n$-dimensional case into the problem (\ref{problem}) for a lower dimensional case.

\begin{thm}\label{main thm of 3.3}
Suppose the similarity condition (\ref{block similarity}) holds. Then,

\begin{enumerate}
\item If $C_i = 0$ for all $i$, then $\coho{1}{G}{{}_{\varphi}\mathcal{T}(K)} \simeq \coho{1}{G}{{}_{\pi}\mathcal{T}(K)}\oplus  \coho{1}{G}{{}_{\psi}\mathcal{T}(K)}$.

\item Suppose $A_i$'s are lower triangluar matrices with only 1 lies on the diagonal or  $A_i$'s are upper triangular matrices with only 1 lies on the diagonal. Then, $\coho{1}{G}{{}_{\varphi}\mathcal{T}(K)} \simeq \coho{1}{G}{{}_{\psi}\mathcal{T}(K)}$.

\item Suppose that every $r_i$ are diagonal matrices, denoted by $r_i = \text{diag}(\,(-1)^{a_{1i}}, (-1)^{a_{2i}},\cdots,$ \\ $(-1)^{a_{ni}} )$ for some $a_{ji}\in \{0,1\}$. Let $\mathcal{U}_i$ be a $G$-integer defined by $g_j \cdot 1 = (-1)^{a_{ij}}$. Then, $\coho{1}{G}{{}_{\varphi}\mathcal{T}(K)} \simeq \bigoplus_{i=1}^{m} \operatorname{Br}(k| K^{I_i})$ where $I_i$ is a canonical subgroup with respect to $\mathcal{U}_i$.
\end{enumerate}
\end{thm}

\begin{proof}
$\phantom{a}$ \newline
\noindent (1) It is easy to check that the identity map from ${}_{\varphi}\mathcal{T}(K)$ to ${}_{\pi}\mathcal{T}(K) \oplus {}_{\psi}\mathcal{T}(K)$ preseves the $G$-action. Therefore, ${}_{\varphi}\mathcal{T}(K) \simeq {}_{\pi}\mathcal{T}(K) \oplus {}_{\psi}\mathcal{T}(K)$ as a $G$-moudle and we have $\coho{1}{G}{{}_{\varphi}\mathcal{T}(K)} \simeq \coho{1}{G}{{}_{\pi}\mathcal{T}(K)} \oplus \coho{1}{G}{{}_{\psi}\mathcal{T}(K)}$. \newline

\noindent(2) Define a projection map $p : {}_{\varphi}\mathcal{T}(K) \rightarrow {}_{\pi}\mathcal{T}(K)$ that sends $(k_1,\cdots,k_n)$ to $(k_1,\cdots,k_t)$. Clearly, $p$ is a surjective $G$-module homomorphism with $\ker{\pi} \simeq {}_{\psi}\mathcal{T}(K)$. So, we have a natual exact sequence of $G$-modules as follows 
\[1 \longrightarrow {}_{\psi}\mathcal{T}(K) \overset{i}\longrightarrow {}_{\varphi}\mathcal{T}(K) \overset{p}\longrightarrow {}_{\pi}\mathcal{T}(K) \longrightarrow 1\]
From (\ref{les module}), this induces a long exact sequence.
\begin{align*}
\coho{0}{G}{{}_{\varphi}\mathcal{T}(K)} & \overset{p^0}{\rightarrow} \coho{0}{G}{{}_{\pi}\mathcal{T}(K)} \rightarrow \coho{1}{G}{{}_{\psi}\mathcal{T}(K)} \overset{i^1}{\rightarrow} \coho{1}{G}{{}_{\varphi}\mathcal{T}(K)} \overset{p^1}{\rightarrow} \coho{1}{G}{{}_{\pi}\mathcal{T}(K)}\rightarrow \cdots
\end{align*}
Take any $(k_1,\cdots,k_t) \in {}_{\pi}\mathcal{T}(K)^G$. Note that $(k_1,\cdots,k_t,1,\cdots,1)$ is a $G$-invariant element of ${}_{\varphi}\mathcal{T}(K)$ that goes to $(k_1,\cdots,k_t) $ via the map $p^0$. So, $p^0$ is surjective and our long exact sequence can be reduced to
\begin{align}\label{third seq}
1 \longrightarrow \coho{1}{G}{{}_{\psi}\mathcal{T}(K)} \longrightarrow \coho{1}{G}{{}_{\varphi}\mathcal{T}(K)} \longrightarrow \coho{1}{G}{{}_{\pi}\mathcal{T}(K)}
\end{align}

Without loss of generality, we may assume that $A_i$'s are all lower triangular matrices. Now, we prove that $\coho{1}{G}{{}_{\pi}\mathcal{T}(K)}$ is trivial by using induction on $t$. When $t=1$, $\coho{1}{G}{K^{\times}}$ is clearly trivial by Hilbert's Theorem 90. For a inductive step, suppose that $t=d$ and $A_i$'s can be partitioned into as follows. \vspace{0.7em}
\[\footnotesize{A_i = \begin{pmatrix}
1 & 0 & \cdots & 0 \\
* & \\
\vdots & & \Huge{D_i} \\
* & 
\end{pmatrix}}
\]

Let $\eta:G \rightarrow \text{GL}(d-1,\mathbb{Z})$ be a group homomorphism that sends $g_i$ to $D_i$. By the same analysis before (\ref{third seq}), we have an exact sequence as $1 \rightarrow \coho{1}{G}{{}_{\eta}\mathcal{T}(K)} \rightarrow \coho{1}{G}{{}_{\pi}\mathcal{T}(K)} \rightarrow \coho{1}{G}{K^{\times}}$.
Since $\coho{1}{G}{K^\times}$ is trivial, we have $\coho{1}{G}{{}_{\pi}\mathcal{T}(K)} \simeq \coho{1}{G}{{}_{\eta}\mathcal{T}(K)}$. By induction hyphothesis, $\coho{1}{G}{{}_{\eta}\mathcal{T}(K)}$ is also trivial and so induction has been finished. From an exact sequence (\ref{third  seq}), we have $\coho{1}{G}{{}_{\varphi}\mathcal{T}(K)} \simeq \coho{1}{G}{{}_{\psi}\mathcal{T}(K)}$.
\newline

\noindent (3) For an fixed integer $i$, consider a group homomorphism $f: G \rightarrow \text{Aut}(\mathbb{Z})\simeq \mathbb{Z}_2=\{1,-1\}$ that sends $g \in G$ to $(i,i)$-th entry of $\varphi(g)$. We can define a $G$-integer $\mathcal{U}_i$ by $g \cdot 1 =f(g)$ for all $g\in G$. So, we checked the well-definedness of each $G$-integer $\mathcal{U}_i$. By applying Theorem \ref{main thm of 3.3}-(1) and Theorem \ref{1-dim}, we have $\coho{1}{G}{{}_{\varphi}\mathcal{T}(K)} \simeq \bigoplus_{i=1}^m \coho{1}{G}{\mathcal{U}_i \otimes K^\times} \simeq \bigoplus_{i=1}^m\brau{k}{K^{I_i}}$.
\end{proof}

\begin{remark}\label{rmk}
Here are concrete descriptions of isomorphisms in Theorem \ref{main thm of 3.3}. 
\begin{itemize}
\item[(1)] $(a,b) \in \coho{1}{G}{{}_{\pi}\mathcal{T}(K)}\oplus \coho{1}{G}{{}_{\psi}\mathcal{T}(K)}$ corresponds to $c \in \coho{1}{G}{{}_{\varphi}\mathcal{T}(K)}$ such that $c_g = (a_g,b_g)$ for all $g \in G$.
\item[(2)] $b \in \coho{1}{G}{{}_{\psi}\mathcal{T}(K)}$ corresponds to $c\in \coho{1}{G}{{}_{\varphi}\mathcal{T}(K)}$ such that $c_g = (1,\cdots,1,b_g)$ for all $g\in G$.
\item[(3)] Denote by $\overline{\phantom{a}}$ the representation of an element of Brauer group. Then, $(\overline{a_1},\cdots,\overline{a_n}) \in  \bigoplus_{i=1}^m \operatorname{Br}(k|K^{I_i})$ corresponds to $c \in \coho{1}{G}{{}_{\varphi}\mathcal{T}(K)}$ such that $c_g = (a_1^{b_1},a_2^{b_2},\cdots, a_n^{b_n})$ where $b_i=0$ when $g \in I_i$ and $b_i=1$ when $g \notin I_i$.
\end{itemize}
\end{remark}

\begin{exa} \label{exa1 of 3.3}
Suppose that $G\simeq \mathbb{Z}_2 \times \mathbb{Z}_2 \leq \operatorname{GL}(4,\mathbb{Z})$ and generated by following matrices.
\[\footnotesize{\left\{ r=\begin{pmatrix} 1&0&0&0 \\ 0&1&0&0 \\ 1&1&-1&0 \\ 0&0&0&1 \end{pmatrix},\,s= \begin{pmatrix} 1&0&0&0 \\ 0&1&0&0 \\ 0&0&1&0 \\ 0&1&0&-1 \end{pmatrix} \right\}}
\]
Let $A_1$, $A_2$, $B_1$, $B_2$ be submatrices as follows.
\[\footnotesize{A_1 = A_2= \begin{pmatrix} 1&0 \\ 0&1\end{pmatrix}, \,\,B_1 = \begin{pmatrix} -1 & 0 \\ 0 & 1 \end{pmatrix}, \,\, B_2 = \begin{pmatrix} 1&0 \\ 0&-1 \end{pmatrix}}
\]
Let $\psi:G \rightarrow \operatorname{GL}(2,\mathbb{Z})$ be a group homomorphism that sends $r$ to $B_1$ and $s$ to $B_2$. Since $A_1$ and $A_2$ are identity, we have $\coho{1}{G}{{}_{\varphi}\mathcal{T}(K)}\simeq \coho{1}{G}{{}_{\psi}\mathcal{T}(K)}$ by Theorem \ref{main thm of 3.3}-(2). We use Theorem \ref{main thm of 3.3}-(3) to compute $\coho{1}{G}{{}_{\psi}\mathcal{T}(K)}$. Let $\mathcal{U}_1$ be a $G$-integer whose $G$-action defined by $r\cdot 1 = -1$, $s \cdot 1 =1$. The canonical subgroup with respect to $\mathcal{U}_1$ is $\left<s\right>$. Let $\mathcal{U}_2$ be a $G$-integer whose $G$-action defined by $r \cdot 1=1$, $s\cdot 1 =-1$. The canonical subgroup with respect to $\mathcal{U}_2$ is $\left<r\right>$. By Theorem \ref{main thm of 3.3}-(3), we have $\coho{1}{G}{{}_{\psi}\mathcal{T}(K)}\simeq \operatorname{Br}(k|K^{\left<s\right>})\oplus \operatorname{Br}(k|K^{\left<r\right>})$.
\end{exa}

\begin{exa} \label{exa2 of 3.3}
Suppose that $G \simeq D_6 \leq \operatorname{GL}(3,\mathbb{Z})$ and generated by following matrices.
\[\footnotesize{\left\{
r= \begin{pmatrix} 1& 0 & 0 \\ 0 & 0 & -1 \\ 0 & 1 & -1 \end{pmatrix}, s = \begin{pmatrix}
-1 & 0 & 0 \\ 0 & 0 & -1 \\ 0 & -1 & 0
\end{pmatrix}  \right\}}
\]
Let $A_1$, $A_2$, $B_1$, $B_2$ be submatrices as follows.
\[
\footnotesize{A_1=1,\,\, A_2=-1, \,\,B_1 = \begin{pmatrix} 0 & -1 \\ 1 & -1 \end{pmatrix}, \,\, B_2 = \begin{pmatrix} 0&-1 \\ -1&0 \end{pmatrix}}
\]
Let $\pi:G \rightarrow \operatorname{Aut}(\mathbb{Z})=\{1,-1\}$ be a group homomorphism that sends $r$ to $A_1$ and $s$ to $A_2$. Let $\psi: G \rightarrow \operatorname{GL}(2,\mathbb{Z})$ be a group homomorphism that sends $r$ to $B_1$ and $s$ to $B_2$.
By Theorem \ref{main thm of 3.3}-(1), $\coho{1}{G}{{}_{\varphi}\mathcal{T}(K)} \simeq \coho{1}{G}{{}_{\pi}\mathcal{T}(K)} \, \oplus \, \coho{1}{G}{{}_{\psi}\mathcal{T}(K)}$. Since $A_i$'s are diagonal matrices, $\coho{1}{G}{{}_{\pi}\mathcal{T}(K)} \simeq \operatorname{Br}(k|K^{\left<r\right>})$ by Theorem \ref{main thm of 3.3}-(3). By Proposition \ref{computation reduce}, we have $\coho{1}{G}{{}_{\psi}\mathcal{T}(K)}\simeq \coho{1}{N}{{}_{\psi}\mathcal{T}(K)}$ where $N$ is a subgroup of $\operatorname{GL}(2,\mathbb{Z})$ generated by $\{B_1,B_2\}$. In Example \ref{induced 2 exa 2}, the last term is computed as $M$ defined in (\ref{M}). Finally, we have $\coho{1}{G}{{}_{\varphi}\mathcal{T}(K)}\simeq \operatorname{Br}(k|K^{\left<r\right>}) \oplus M$.

Now, we find generators of $\coho{1}{G}{{}_{\varphi}\mathcal{T}(K)}$. By Remark \ref{rmk}-(3), $\coho{1}{G}{{}_{\pi}\mathcal{T}(K)}$ is generated by $\{c(a) \, | \, a \in k^\times \}$ where $c(a)_r = 1$ and $c(a)_s = a$. In Example \ref{induced 2 exa 2}, we showed that $\coho{1}{G}{{}_{\psi}\mathcal{T}(K)}$ is generated by $\{c^{\prime}(b)\, | \,b\in K^\times \, \text{with}\,\,r(b)=b,\,s(b)=b^{-1} \}$ where $c^{\prime}(b)_{r} = (1,b^{-1})$ and $c^{\prime}(b)_{s} = (1,1)$. By Remark \ref{rmk}-(1), $\coho{1}{G}{{}_{\varphi}\mathcal{T}(K)}$ is generated by the set of 1-cocycles $\{c(a,b)\, | \, a\in k^{\times},\, b\in K^\times \, \text{with}\,\,r(b)=b,\, s(b)=b^{-1} \}$ where $c(a,b)_r = (1,1,b^{-1})$ and $c(a,b)_s = (a,1,1)$.
\end{exa}

See Table 7 of Appendix A. When $n=2$ or $n=3$, we summarized a conjugacy classes of $G$ in $\text{GL}(n,\mathbb{Z})$ that can be applied Theorem \ref{main thm of 3.3}. For these cases, $\coho{1}{G}{{}_{\varphi}\mathcal{T}(K)}$ can be computed as Example \ref{exa1 of 3.3} and \ref{exa2 of 3.3}.

\section{Some Classification of Twisted Form of Toric Variety}
For a cone $\sigma \subset N$, denote by $\sigma(1)$ the set of rays of $\sigma$. Also, denote by $\Sigma(1)$ the set of rays of a fan $\Sigma$. Unlike Section 3, we remove the assumption that $G$ is a subgroup of $\text{GL}(n,\mathbb{Z})$. In this section, we reveal the structure of $\coho{1}{G}{\text{Aut}_{\Sigma}^T}$ and classify $K/k$-forms of a quasi-projective toric variety $X_{\Sigma}$. This work begins by proving some useful results for $\text{Aut}_{\Sigma}$.

\begin{lem} 
A split toric variety $X_{\Sigma}$ has no torus factor if and only if $\operatorname{Aut}_{\Sigma}$ is a finite group.
\end{lem}
\begin{proof}
Let $t$ be the number of rays in $\Sigma$ and $\Sigma(1)=\{u_i\,|\,i=1,\cdots,t\}$. It is well-known that $X_{\Sigma}$ has a torus factor if and only if $\Sigma(1)$ does not spans $\mathbb{Z}^n$ (\cite{cox}, Proposition 3.3.9). Take any $A\in \text{Aut}_{\Sigma}$. Since $A$ is a bijection from $\Sigma(1)$ to $\Sigma(1)$, there exist $\sigma \in S_t$ such that $Au_i = u_{\sigma(i)}$ for all $1\leq i \leq t$. When $X_{\Sigma}$ has no torus factor, such matrix $A$ uniquely exist since $\Sigma(1)$ spans $\mathbb{Z}^n$. So, $\text{Aut}_{\Sigma}$ is a subgroup of $S_t$. When $X_{\Sigma}$ has a torus factor, such $A$ infinitely many exists since $\Sigma(1)$ does not span $\mathbb{Z}^n$. So, $\text{Aut}_{\Sigma}$ is infinite.
\end{proof}

\begin{thm} \label{automor thm}
Let $\sigma \subset \mathbb{Z}^n$ be a strongly convex cone and $\Sigma$ be a fan consists of all faces of $\sigma$. Then, $\operatorname{Aut}_{\Sigma}$ is a subgroup of $D_{2t}$ where $t$ is the number of codimension 2 faces of $\sigma$.
\end{thm}
\begin{proof}
In this proof, we call a codimension 1 face of $\sigma$ as a facet of $\sigma$. Let $S=\{v_i \, | \, i=1,\cdots,t\}$ be the set of codimension 2 faces of $\sigma$. Define a plane graph $\mathcal{G}$ on $S$ by connecting $v_i$ and $v_j$ if $v_i(1)\cup v_j(1)$ generates some facet of $\sigma$. 

Let $\tau$ be an arbitrary facet of $\sigma$. Since $\tau$ is simplicial (\cite{polytope}, Proposition 2.16), any subset of $\tau(1)$ generate a face of $\tau$. For any $v_i \in S$, there exists some $j \in \{1,\cdots,m\}$ such that $v_i(1)\cup v_j(1)=\tau(1)$. This implies that there is no isolated vertex of a graph $\mathcal{G}$. It is well-known that each codimension 2 face is an intersection of exactly two facets (\cite{combinatorial}, Theorem 1.10). This implies that each vertex of a graph $\mathcal{G}$ has exactly two edges. Therefore, the graph $\mathcal{G}$ becomes a $t$-gon. Since any element of $\text{Aut}_{\Sigma}$ preserves the symmetry of this $t$-gon, $\text{Aut}_{\Sigma}$ is a subgroup of $D_{2t}$.
\end{proof}

When $X_{\Sigma}$ is affine, we can reduce the candidates of $\text{Aut}_{\Sigma}$ by applying previous theorem..

\begin{cor}\label{automor cor}
Suppose that $X_{\Sigma}$ is an affine toric variety with no torus factor. If $n=2$ or $n=3$, then $\operatorname{Aut}_{\Sigma}$ is isomorphic to one of $\mathbb{Z}_2$, $\mathbb{Z}_3$, $\mathbb{Z}_4$, $\mathbb{Z}_6$, $D_4$, $D_6$, $D_8$, $D_{12}$ or trivial group.
\end{cor}
\begin{proof}
See the Table 2 of Appendix A and Tahars's paper\cite{tahara}. We see that $D_4$, $D_6$, $D_8$ and $D_{12}$ are only possible diheral subgroups of $\text{GL}(2,\mathbb{Z})$ and $\text{GL}(3,\mathbb{Z})$. By applying Theorem \ref{automor thm}, $\text{Aut}_{\Sigma}$ is isomorphic to one of the subgroup of $D_4$, $D_6$, $D_8$, $D_{12}$. The lists of subgroups of $D_4$, $D_6$, $D_8$ and $D_{12}$ are given as $\mathbb{Z}_2$, $\mathbb{Z}_3$, $\mathbb{Z}_4$, $\mathbb{Z}_6$, $D_4$, $D_6$, $D_8$, $D_{12}$ or trivial group. 
\end{proof}

\begin{cor}\label{automor cor2}
Let $\sigma \subset \mathbb{Z}^3$ be a 3-dimensional cone that has $t$ rays and $\Sigma$ be a fan consists of all faces of $\sigma$. If $t$ is an odd number, $\text{Aut}_{\Sigma}$ is isomorphic to one of $\mathbb{Z}_2$, $\mathbb{Z}_3$, $D_6$ or trivial group. If $t=3$, then $\operatorname{Aut}_{\Sigma}$ is exactly isomorphic to $D_6$.
\end{cor}
\begin{proof}
By Theorem \ref{automor thm}, $\text{Aut}_{\Sigma}$ is isomorphic to one of the subgroup of $D_{2t}$. Also, the previous corollary implies that $\text{Aut}_{\Sigma}$ is isomorphic to one of $\mathbb{Z}_2$, $\mathbb{Z}_3$, $\mathbb{Z}_4$, $\mathbb{Z}_6$, $D_4$, $D_6$, $D_8$, $D_{12}$ or trivial group. If $t=1$, possible candidates of $\text{Aut}_{\Sigma}$ is a trivial group or $\mathbb{Z}_2$. Suppose that $t\geq 3$. There is no element of order 4 and order 6 in $D_{2t}$. So, $\mathbb{Z}_4$, $\mathbb{Z}_6$, $D_8$ and $D_{12}$ are excluded from the lists of the candidates of $\text{Aut}_{\Sigma}$. Since $D_4$ is not a subgroup of $D_{2t}$, $D_4$ is also excluded. Hence, $\text{Aut}_{\Sigma}$ is isomorphic to one of $\mathbb{Z}_2$, $\mathbb{Z}_3$ and $D_6$. Now, consider the case $t=3$. Since $\sigma$ is 3-dimensional, the rays of $\sigma$ form a $\mathbb{Z}$-basis of $\mathbb{Z}^3$. Hence, $\text{Aut}_{\Sigma}$ is same as the permutation of these rays, so it is isomorphic to $D_6$.
\end{proof}

Now, we suggest the algorithm for calculating $\coho{1}{G}{\text{Aut}_{\Sigma}^T}$ as follows.

\MyQuote{\noindent $<$ An algorithm to compute $\coho{1}{G}{\text{Aut}_{\Sigma}^T}$\,$>$
\newline Let $L$ be an intermediate field $k\subset L \subset K$ fixed by a subgroup $\ker{\varphi}\subset G$, $\overline{G}=G/\ker{\varphi} \simeq \text{Gal}(L/k)$ and $\overline{\varphi}:\overline{G} \rightarrow \text{Aut}_{\Sigma}\subset \text{GL}(n,\mathbb{Z})$ be an injective morphism induced by $\varphi$. We treat $\overline{G}$ as a subgroup of $\text{GL}(n,\mathbb{Z})$ by considering each element $g\in \overline{G}$ is a matrix $\overline{\varphi}(g) \in \text{GL}(n,\mathbb{Z})$.

\begin{enumerate}
\item Check the similarity condition of Theorem \ref{induced 1}, \ref{induced 2}, \ref{main thm of 3.3} for a Galois group $\overline{G}$.
\item Identify the form of generators of $\coho{1}{\overline{G}}{{}_{\overline{\varphi}}\mathcal{T}(L)}$ from Corollary \ref{generator form}, Theorem \ref{induced 2}-(2) and Remark \ref{rmk}.
\item Identify $C_{\text{Aut}_{\Sigma}}(\varphi(G))$-action from (\ref{matrix action}) and find $\coho{1}{\overline{G}}{{}_{\overline{\varphi}}\mathcal{T}(L)}/C_{\text{Aut}_{\Sigma}}(\varphi(G))$.
\item By taking disjoint union of these terms over all representatives of conjugacy classes of $\varphi$, we obtain $\coho{1}{G}{\text{Aut}_{\Sigma}^T}$
\end{enumerate}
}\label{algorithm}

In step (3), $\coho{1}{\overline{G}}{{}_{\overline{\varphi}}\mathcal{T}(L)}/C_{\text{Aut}_{\Sigma}}(\varphi(G))$ is isomorphic to $\coho{1}{G}{{}_{\varphi}\mathcal{T}(K)}/C_{\text{Aut}_{\Sigma}}(\varphi(G))$ by Proposition \ref{computation reduce}. The step (1) and (2) can be completed as examples of Section 3. The step (3) and (4) are not difficult, just simple calculation is needed. Here is the number of conjugacy classes of $\overline{G} \subset \text{GL}(n,\mathbb{Z})$ that can be applied Theorem \ref{induced 1}, \ref{induced 2} and \ref{main thm of 3.3}, see Appendix 5.
\begin{enumerate}
\item[$\cdot$] When $n=2$, all conjugacy classes of $\text{GL}(2,\mathbb{Z})$.
\item[$\cdot$] When $n=3$, 61 cases among 73 conjugacy classes of $\text{GL}(3,\mathbb{Z})$.\end{enumerate}

Since there are too many cases in above lists, we show the algorithm only when $n=3$ and $\text{Aut}_{\Sigma}\simeq D_6$. Other cases can be computed similarly. This is the table for conjugacy classes of subgroups of $\text{GL}(3,\mathbb{Z})$ which is isomorphic to $D_6$, see Tahara\cite{tahara}. The first matrix is a rotation and the second matrix is a relection of $D_6$. 
\begin{table}[ht]
\footnotesize{
\begin{tabular}{l|l}
\hline
\multirow{4}*{$W_5 =\left\{\begin{pmatrix} 1 & 0 & 0 \\ 0 & 0&-1 \\ 0&1&-1 \end{pmatrix} ,\begin{pmatrix} -1&0&0 \\ 0&0&-1 \\ 0&-1&0 \end{pmatrix} \right\}$} & \multirow{4}*{$W_6 =\left\{\begin{pmatrix} 1 & 0 & 0 \\ 0 & 0&-1 \\ 0&1&-1 \end{pmatrix} ,\begin{pmatrix} 1&0&0 \\ 0&0&1 \\ 0&1&0 \end{pmatrix} \right\}$} \\[3.5em] \hline
\multirow{4}*{$W_7 =\left\{\begin{pmatrix} 1 & 0 & 0 \\ 0 & 0&-1 \\ 0&1&-1 \end{pmatrix} ,\begin{pmatrix} -1&0&0 \\ 0&0&1 \\ 0&1&0 \end{pmatrix} \right\}$} & \multirow{4}*{$W_8 =\left\{\begin{pmatrix} 1 & 0 & 0 \\ 0 & 0&-1 \\ 0&1&-1 \end{pmatrix} ,\begin{pmatrix} 1&0&0 \\ 0&0&-1 \\ 0&-1&0 \end{pmatrix} \right\}$} \\[3.5em] \hline
\multirow{4}*{$W_9 =\left\{\begin{pmatrix} 0 & 1 & 0 \\ 0 & 0&1 \\ 1&0&0 \end{pmatrix} ,\begin{pmatrix} 0&0&-1 \\ 0&-1&0 \\ -1&0&0 \end{pmatrix} \right\}$} & \multirow{4}*{$W_{10} =\left\{\begin{pmatrix} 0 & 1 & 0 \\ 0 & 0&1 \\ 1&0&0 \end{pmatrix} ,\begin{pmatrix} 0&0&1 \\ 0&1&0 \\ 1&0&0 \end{pmatrix} \right\}$} \\[3.5em] \hline
\end{tabular}
\caption{Conjugacy classes of $\text{GL}(3,\mathbb{Z})$ which is isomorphic to $D_6$}}
\end{table} 

From now on, we denote $\coho{1}{\overline{G}}{{}_{\overline{\varphi}}\mathcal{T}(L)}/C_{\text{Aut}_{\Sigma}}(\varphi(G))$ simply as $H^1/H^0$.

\begin{thm} \label{D_6}
Suppose that $\varphi:G\rightarrow \operatorname{Aut}_{\Sigma}\subset \operatorname{GL}(3,\mathbb{Z})$ is a group homomorphism with $\operatorname{Aut}_{\Sigma} \simeq D_6$. Denote by $L$ the intermediate field $k \subset L \subset K$ which is fixed by a subgroup $\ker{\varphi}\subset G$. The table for $H^1/H^0$ is given as follows.

\begin{center}
\footnotesize{
\begin{tabular}{c|c|c|c|c|c|c|c|c}
\hline
$\operatorname{Aut}_{\Sigma}$ & $\varphi(G)$ & $H^1/H^0$ &$\operatorname{Aut}_{\Sigma}$ & $\varphi(G)$ & $H^1/H^0$ & $\operatorname{Aut}_{\Sigma}$ & $\varphi(G)$ & $H^1/H^0$ \\ \hline \hline
\multirow{3}*{$W_5$} & $\mathbb{Z}_2$ & $\operatorname{Br}(k|L)$ & \multirow{3}*{$W_7$} &$\mathbb{Z}_2$ & $\operatorname{Br}(k|L)$ & \multirow{3}*{$W_9$} & $\mathbb{Z}_2$ & $\operatorname{Br}(k|L)$  \\ \cline{2-3} \cline{5-6} \cline{8-9}
& $\mathbb{Z}_3$ & $\operatorname{Br}(k|L)$ & & $\mathbb{Z}_3$ & $\operatorname{Br}(k|L)$ & & $\mathbb{Z}_3$ & $\operatorname{trivial}$ \\ \cline{2-3} \cline{5-6} \cline{8-9}
& $D_6$ & $\operatorname{Br}(k|L^{\left<s\right>}) \oplus \operatorname{Br}(k|L^{\left<r\right>})$ & &$D_6$ & $M\oplus \operatorname{Br}(k|L^{\left<r\right>})$ & & $D_6$ & $\operatorname{Br}(L^{\left<sr\right>}|L)$ \\ \hline
\multirow{3}*{$W_6$} & $\mathbb{Z}_2$ & $\operatorname{trivial}$ & \multirow{3}*{$W_8$} &$\mathbb{Z}_2$ & $\operatorname{trivial}$ & \multirow{3}*{$W_{10}$} & $\mathbb{Z}_2$ & $\operatorname{trivial}$ \\ \cline{2-3} \cline{5-6} \cline{8-9}
& $\mathbb{Z}_3$ & $\operatorname{Br}(k|L)$ & & $\mathbb{Z}_3$ & $\operatorname{Br}(k|L)$ & & $\mathbb{Z}_3$ & $\operatorname{trivial}$ \\ \cline{2-3} \cline{5-6} \cline{8-9}
& $D_6$ & $M$ & & $D_6$ & $\operatorname{Br}(k|L^{\left<s\right>})$ & & $D_6$ & $\operatorname{trivial}$ \\ \hline
\end{tabular}}
\end{center}
When $\varphi(G) \simeq D_6$, let $r,s$ be any elements of $G$ such that $\varphi(r)$ has order 3 and $\varphi(s)$ has order 2. The group $M$ is defined as follows.
\[M = \frac{\{a\in L^\times \, | \, r(a)=a,s(a)=a^{-1}\}}{\{\norm{L}{L^{\left<r\right>}}{(b)}\,|\,b \in L^\times \,\, \text{with} \,\, \norm{L}{L^{\left<sr\right>}}{(b)}=1\}}\]
\end{thm}
\begin{proof}
By the definition of $L$, the fields $L^{\left<r\right>}$, $L^{\left<s\right>}$ and $L^{\left<sr\right>}$ is not changed by the choice of $r$ and $s$. Denote $\overline{G}=G/\ker{\varphi}\simeq \varphi(G)\simeq \text{Gal}(L/k)$ and $\overline{\varphi}:\overline{G} \rightarrow \text{Aut}_{\Sigma}\subset \text{GL}(3,\mathbb{Z})$ be an injective group homomorphism induced by $\varphi$. We treat $\overline{G}$ as a subgroup of $\text{Aut}_{\Sigma}$ by considering each $g \in \overline{G}$ is a matrix $\overline{\varphi}(g) \in \text{GL}(3,\mathbb{Z})$. Since $H^1/H^0$ is invariant under the conjugacy classes of $\varphi$, we may assume that $\text{Aut}_{\Sigma}\subset \text{GL}(3,\mathbb{Z})$ is generated by one of the Table 1. For a notational convenience, we write $\varphi(r)$ and $\varphi(s)$ as just $r$ and $s$, respectively. Also, $\cdot$ denotes the $G$-action defined as (\ref{G-action}) and $*$ as a $\text{GL}(3,\mathbb{Z})$-action defined as (\ref{matrix action}). The step (1) and (2) of the algorithm (\ref{algorithm}) are introduced in Appendix B. In this proof, we use results of Appendix B and show only the step (3).

First, consider the case $\varphi(G) \simeq \mathbb{Z}_2$. Since all subgroups of $D_6$ of order 2 are conjugate each other, we may assume $\overline{G} =\{1,s\}$. When $\text{Aut}_{\Sigma}$ is generated by one of $W_6$, $W_8$ or $W_{10}$, $\coho{1}{\overline{G}}{{}_{\overline{\varphi}}\mathcal{T}(L)}$ is trivial by Lemma \ref{lem 1}. When $\text{Aut}_{\Sigma}$ is generated by one of $W_5$, $W_6$ or $W_7$, $\coho{1}{\overline{G}}{{}_{\overline{\varphi}}\mathcal{T}(L)} \simeq \brau{k}{L}$ and generators are written as $c(a)_s = (a,1,1)$ for each $a\in k^\times$. Note that $C_{\text{Aut}_{\Sigma}}(\varphi(G))=\{1,s\}$ and $t^{-1}c(a)_s (s \cdot t) = (s* c(a))_s$ for $t=(a,1,1)$. It means that $c(a)$ and $s*c(a)$ are cohomologous. Hence, there is only one $C_{\text{Aut}_{\Sigma}}(\varphi(G))$-orbits and so $H^1/H^0 \simeq \brau{k}{L}$. Next, consider the case $\varphi(G)\simeq \mathbb{Z}_3$. Since all subgroups of $D_6$ of order 3 are conjugate each other, we may assume $\overline{G} = \{1,r,r^2\}$. When $\text{Aut}_{\Sigma}$ is generated by $W_9$ or $W_{10}$, $\coho{1}{\overline{G}}{{}_{\overline{\varphi}}\mathcal{T}(L)}$ is trivial by Lemma \ref{lem 2}. When $\text{Aut}_{\Sigma}$ is generated by $W_5$,$W_6$,$W_7$ or $W_8$, $\coho{1}{\overline{G}}{{}_{\overline{\varphi}}\mathcal{T}(L)} \simeq \brau{k}{L}$ and generators are written as $c(a)_r=(1,1,a^{-1})$ for each $a \in k^\times$. Note that $C_{\text{Aut}_{\Sigma}}(\varphi(G)) = \{1,r,r^2\}$ and $t^{-1}c(a)_r(r\cdot t) = (r*c(a))_r$ for  $t=(1,1,a^{-1})$. So, there is only one $C_{\text{Aut}_{\Sigma}}(\varphi(G))$-orbit and $H^1/H^0\simeq \brau{K}{L}$. Finally, consider when $\varphi(G)\simeq D_6$. Since $C_{\text{Aut}_{\Sigma}}(\varphi(G))$ is trivial, $H^1/H^0$ is same as $\coho{1}{\overline{G}}{{}_{\overline{\varphi}}\mathcal{T}(L)}$ and we have results from Lemma \ref{lem 3}.
\end{proof}

Now, we can compute the number of isomorphism classes of $K/k$-forms of $X_{\Sigma}$. By Corollary \ref{automor cor2}, the table of Theorem \ref{D_6} classifies all $K/k$-forms of 3-dimensional affine toric varieties associated to a cone $\sigma \subset \mathbb{Z}^3$ generated by 3 rays. Here are the concrete description of step (4).

\begin{exa} \label{exa k-form}
Let $\sigma \subset \mathbb{Z}^3$ be a cone generated by $\{e_1=(1,0,-1),e_2=(0,-1,1),e_3=(-1,1,0)\}$ and $\Sigma$ be a fan consists of all faces of $\sigma$. Since all matrices of $W_9$ in Table 1 preserves a fan $\Sigma$, $\operatorname{Aut}_{\Sigma}\simeq D_6$ and generated by matrices of $W_9$. Let $r$ be the first matrix of $W_9$ and $s$ be the second matrix of $W_9$.

\begin{itemize}
\item[(1)] $k=\mathbb{R}$, $K=\mathbb{C}$
\end{itemize}
In this case, $\operatorname{Gal}(\mathbb{C}/\mathbb{R})=\{\text{id},c\}$ where $c$ is a conjugation map. There are two conjugacy classes of $\varphi:G \rightarrow \operatorname{Aut}_{\Sigma}\subset \operatorname{GL}(3,\mathbb{Z})$. When $\varphi(G)$ is trivial, clearly $H^1/H^0$ is trivial. When $\varphi(G)\simeq \mathbb{Z}_2$, we have $H^1/H^0 \simeq \operatorname{Br}(\mathbb{R}|\mathbb{C})\simeq \mathbb{Z}_2$ by Theorem \ref{D_6}. Hence,
\[
\{\text{isomorphism classes of $\mathbb{C}/\mathbb{R}$-forms of}\,\, X_{\Sigma}\,\} \simeq \{\cdot\} \amalg \operatorname{Br}(\mathbb{R}|\mathbb{C}) \simeq \{\cdot\}\amalg \{1,-1\}
\]
and there are 3 isomorphism classes of $\mathbb{C}/\mathbb{R}$-forms of $X_{\Sigma}$.

\begin{itemize}
\item[(2)] $k=\mathbb{Z}_2$, $K=\mathbb{Z}_2(w)$ where $w$ is an primitive 5th root of unity.
\end{itemize}
In this case, $G=\operatorname{Gal}(K/k)$ is a cyclic group of order 4. Let $c$ be a generator of $G$ defined by $c(w)=w^2$. There are two conjugacy classes of $\varphi:G \rightarrow \operatorname{Aut}_{\Sigma}\subset \operatorname{GL}(3,\mathbb{Z})$. When $\varphi(G)$ is trivial, clearly $H^1/H^0$ is trivial. Otherwise, $\varphi(c)=s$ and $\varphi(G)\simeq \mathbb{Z}_2$. By Theorem \ref{D_6}, we have $H^1/H^0 \simeq \operatorname{Br}(\mathbb{Z}_2|\mathbb{Z}_2^{\left<c^2\right>}) \simeq 1$. Hence, there are 2 isomorphism classes of $K/k$-forms of $X_{\Sigma}$. 
\begin{align*}
\{\text{isomorphism classes of $K/k$-forms of}\,\, X_{\Sigma}\,\}  \simeq \{\cdot\} \amalg  \{\cdot\}
\end{align*}
\end{exa}

\begin{exa}
Let $\sigma \subset \mathbb{Z}^3$ be a cone generated by $\{e_1=(1,0,1),e_2=(0,1,1),e_3=(1,1,0)\}$ and $\Sigma$ be a fan consists of all faces of $\sigma$. Since all matrices of $W_{10}$ in Table 1 preseves a fan $\Sigma$, $\operatorname{Aut}_{\Sigma}\simeq D_6$ and generated by matrices of $W_{10}$.

\begin{itemize}
\item[(1)] When $K/k$ is a degree 2 extension.
\end{itemize}
In this case, $\operatorname{Gal}(K/k)$ is isomorphic to $\mathbb{Z}_2$. There are two conjugacy classes of $\varphi$. When $\varphi(G)$ is trivial, $H^1/H^0$ is clearly trivial. When $\varphi(G)\simeq \mathbb{Z}_2$, $H^1/H^0$ is also trivial by Theorem \ref{D_6}. Hence,
\begin{align*}
\{ \text{isomorphism classes of $K/k$-forms of}\,\, X_{\Sigma}\,\} \simeq \{\cdot\}\amalg \{\cdot\}
\end{align*}
and there are 2 isomorphism classses of $K/k$-forms of $X_{\Sigma}$.

\begin{itemize}
\item[(2)] When $K/k$ is a degree 3 extension
\end{itemize}
In this case, $G=\operatorname{Gal}(K/k)$ is a cyclic group of order 3. Let $c$ be the generators of $G$. There are two conjugacy classes of $\varphi:G \rightarrow \operatorname{Aut}_{\Sigma}\subset \operatorname{GL}(3,\mathbb{Z})$. When $\varphi(G)$ is trivial, $H^1/H^0$ is a trivial group. Otherwise, $\varphi(G)\simeq \mathbb{Z}_3$ and $H^1/H^0$ is also trivial by Theorem \ref{D_6}. Hence, we have
\begin{align*}
\{\text{isomorphism classes of $K/k$-forms of}\,\, X_{\Sigma}\,\} \simeq \{\cdot\}\amalg \{\cdot\}
\end{align*}
\end{exa}

Now, we verify which $K/k$-forms of which split toric varieties can be classified via the algorithm (\ref{algorithm}). We follow the notations $L$, $\overline{G}$, $\overline{\varphi}$ described in (\ref{algorithm}). To see whether the algorithm (\ref{algorithm}) is applicable, we should check whether the conjugacy class of $\overline{G}$ is contained in Appendix A.

First, we give some restrictions to the group $G=\text{Gal}(K/k)$. Consider when $\overline{G}$ is one of the finite subgroup of $\text{GL}(2,\mathbb{Z})$. Since all conjugacy classes of all finite subgroups of $\text{GL}(2,\mathbb{Z})$ are covered by Appendix A, the algorithm (\ref{algorithm}) can be applied for any 2-dimensional split toric varieties. Consider when $\overline{G}$ is one of the finite subgroup of $\text{GL}(3,\mathbb{Z})$. All conjugacy classes of all finite cyclic groups of $\text{GL}(3,\mathbb{Z})$ are covered by Appendix 5. If $G$ is a finite cyclic group, $\overline{G}$ is also finite cyclic subgroup of $\text{GL}(3,\mathbb{Z})$. So, the algorithm (\ref{algorithm}) can be applied for all quasi-projective 3-dimensional split toric varieties when $K/k$ is a finite cyclic extension.

Now, we give some restrictions to the fan $\Sigma$ in $\mathbb{Z}^3$. Suppose that $\text{Aut}_{\Sigma}$ is a finite cyclic group or a dihedral group. Then, $\overline{G}\simeq \varphi(G) \subset \text{Aut}_{\Sigma}$ is also a finite cyclic group or a dihedral group. All conjugacy classes of finite cyclic groups and dihedral groups of $\text{GL}(3,\mathbb{Z})$ are covered by Appendix 5. Hence, the algorithm (\ref{algorithm}) can be applied for all quasi-projective 3-dimensional split toric varieties $X_{\Sigma}$ when $\text{Aut}_{\Sigma}$ is a finite cyclic or a dihedral group. By Corollary \ref{automor cor}, all 3-dimensional affine toric varieties with no torus factor belongs to these cases. Here are $K/k$-forms that can be classified via the algorithm (\ref{algorithm}).
\begin{enumerate}
\item[$\cdot$] $K/k$-forms of all split toric surfaces for any finite Galois extension $K/k$.
\item[$\cdot$] $K/k$-forms of all 3-dimensional quasi-projective split toric varieties when $K/k$ is a finite cyclic extension.
\item[$\cdot$] $K/k$-forms of all 3-dimensional quasi-projective split toric varieties $X_{\Sigma}$ such that $\text{Aut}_{\Sigma}$ is a finite cyclic group or a dihedral group for any finite Galois extension $K/k$. In particular, all 3-dimensional affine toric varieties with no torus factor belongs to these cases.
\end{enumerate}

Although we didn't check for conjugacy classes of $\text{GL}(n,\mathbb{Z})$ when $n\geq 4$, we expect that there are some cases that can be applied the algorithm (\ref{algorithm}). \newpage

\section*{Appendix A}
\subsection{Conjugacy classes of $\text{GL}(n,\mathbb{Z})$ that can be applied Theroems \ref{induced 1}, \ref{induced 2}, \ref{main thm of 3.3}}

As in Section 3, we assume that $G$ is a subgroup of $\text{GL}(n,\mathbb{Z})$. In this section, we summarized a conjugacy class of $G$ in $\text{GL}(n,\mathbb{Z})$, a subgroup $N$ of $G$ and an $N$-integer $\mathcal{U}$ that satisfy Theorem \ref{induced 1}, \ref{induced 2}, \ref{main thm of 3.3}. A classification of conjugacy classes of a finite subgroup of $\text{GL}(2,\mathbb{Z})$ has already been completed. This work was introduced by Morris Newman in his book\cite{newman} as follows.

\begin{table}[ht]
\footnotesize{
\begin{tabular}{ll|llll}
\hline
Generators  && Groups &  \\ \hline\hline &&& \\[-.8em]
\multirow{2}*{$A= \begin{pmatrix} 0&-1 \\ 1&1 \end{pmatrix}$}& \multirow{2}*{$B = \begin{pmatrix} 0 & -1 \\ 1 & 0 \end{pmatrix}$} & $C_1 = \left<I\right>$ & $C_{2} = \left<-I\right>$ & $D_2 = \left<C\right>$ & $D_2^{\prime} = \left<J\right>$ \\ 
& & $C_3 = \left<A^2\right>$ & $C_4 = \left<B\right>$ & $C_4 = \left<B\right>$ & $D_4 = \left<-I,C\right>$  \\[0.5em] 
\multirow{2}*{$C = \begin{pmatrix} 1 & 0 \\ 0 & -1 \end{pmatrix}$} & \multirow{2}*{$J = \begin{pmatrix} 0 & 1 \\ 1 & 0 \end{pmatrix}$} & $D_4^{\prime} = \left<J, -I \right>$ & $C_6 = \left<A\right>$ & $D_6 = \left<A^2,JA\right>$ & $D_6^{\prime} = \left<A^2,J\right>$ \\ 
&&$D_8 = \left<B,J\right>$ & $D_{12} = \left<A,J\right>$ \\[0.5em] \hline
\end{tabular}}
\caption{Finite subgroups of $\text{GL}(2,\mathbb{Z})$ and their generators}
\end{table} \vspace{-1em}
\noindent A classification of conjugacy classes of finite subgroups of $\text{GL}(3,\mathbb{Z})$ also has been completed by Tahara. Since there are too many cases, we use the notation $W_i$ appeared in \cite{tahara} instead of showing all lists. For each $g\in G$, $T_g$ denotes the linear operator defined in (\ref{def:operator}).

First, consider when $G$ is one of the groups in Table 2. In this case, $g_i$ denotes the $i$-th generator of $G$ appeared in Table 2 with keeping order. For example, $g_1=B$ and $g_2=J$ when $G=D_{8}$, $g_1=J$ and $g_2=-I$ when $G=D_4^{\prime}$. Next, consider when $G$ is one of the groups in Tahara's list\cite{tahara}. In this case, $g_i$ denotes the $i$-th generator of $G$ appeared in Tahara's list with keeping order. For example, when $G= D_6$ and its conjugacy class is $W_5$, $g_1$ and $g_2$ are given as follows, see Table 1.
\[
g_1 = \footnotesize{\begin{pmatrix} 1&0&0 \\ 0&0&-1 \\ 0&1&-1 \end{pmatrix}}, \,\, g_2 = \begin{pmatrix} -1&0&0 \\ 0&0&-1 \\ 0&-1&0 \end{pmatrix}
\] \newline
\noindent (1) The cases of Theorem \ref{induced 1}

First, consider when $G$ is a subgorup of $\text{GL}(2,\mathbb{Z})$. Table 3 summarize a conjugacy class of $G$, a subgroup $N$ of $G$ and an $N$-integer $\mathcal{U}$ that satisfy Theorem \ref{induced 1}. For each cases, the set $\mathcal{B}=\{e_1=1\otimes 1, e_2=g_1\otimes 1\}$ is a $\mathbb{Z}$-basis for $\ind{N}{G}{\mathcal{U}}$. By considering $T_{g_i}$ as its matrix representation with respect to $\mathcal{B}$, we can check $T_{g_i} = g_i$ for all $i$.

\begin{table}[ht]
\footnotesize{
\begin{tabular}{c|c|c}
\hline
$G$ & $N$ & $\mathcal{U}$ \\ \hline \hline
$D_8$ & $\left<g_1^{\,2}, \, g_2g_1\right>$ & $g_2g_1 \cdot 1 =1$, $g_1^{\,2} \cdot 1 = -1$ \\ \hline
$D_4^{\prime}$ & $\left<g_2\right>$ & $g_2 \cdot 1 = -1$  \\ \hline
$C_4$ & $\left<g_1^{\,2}\right>$ & $g_1^{\,2} \cdot 1 = -1$ \\ \hline
$D_2^{\prime}$ & trivial & trivial \\ \hline
\end{tabular}}
\caption{Conjugacy classes in $\text{GL}(2,\mathbb{Z})$ that can be applied Theorem \ref{induced 1}}
\end{table}
Next, consider when $G$ is a subgroup of $\text{GL}(3,\mathbb{Z})$. Table 4 summarize a conjugacy class of $G$, a subgroup $N$ of $G$, an $N$-integer $\mathcal{U}$ and $T \in \text{GL}(3,\mathbb{Z})$ that satisfy Theorem \ref{induced 1}. For $G=\mathbb{Z}_6,\, D_6, \, D_{12}, \, A_4, \, A_4 \times \mathbb{Z}_2$, the set $\mathcal{B}= \{e_1= 1\otimes 1, e_2 = g_1\otimes 1 ,e_3 = g_1^2\otimes 1\}$ is a $\mathbb{Z}$-basis for $\ind{N}{G}{\mathcal{U}}$. For $G=S_4,\, S_4\times \mathbb{Z}_2$, the set $\mathcal{B}=\{e_1=1\otimes 1, \, e_2= t\otimes 1,\, e_3 =t^2\otimes 1 \}$ is a $\mathbb{Z}$-basis for $\ind{N}{G}{\mathcal{U}}$ where $t=g_1^{\,3}g_2g_1^{\,2}$. By considering $T_{g_i}$ as its matrix representation with respect to $\mathcal{B}$, we can check $T_{g_i}=T^{-1}g_iT$ for all $i$.

\begin{table}[ht]
\footnotesize{
\begin{tabular}{c|c|c|c|c}
\hline
$G$ & conj class & $N$ & $\mathcal{U}$ & $T$ \\ \hline \hline &&&& \\[-.8em]
$\mathbb{Z}_6$ & $W_4$ & $\left<g_1^{\,3}\right>$ & $g_1^{\,3} \cdot 1 = -1$ & $\footnotesize{\begin{pmatrix} 0&0&1 \\ 0&-1&0 \\ 1&0&0 \end{pmatrix}}$ \\[2em] \hline  &&&& \\[-.8em]
\multirow{4}*{$D_6$} & $W_9$ & $\left<g_2\right>$ & $g_2 \cdot 1 =-1$ & $\footnotesize{ \begin{pmatrix} 0&1&0 \\ 1&0&0 \\ 0&0&1 \end{pmatrix}}$ \\[2em] \cline{2-5} \\[-.8em]
& $W_{10}$ & $\left<g_2\right>$ & $g_2 \cdot 1 =1$  & $\footnotesize{\begin{pmatrix} 0&1&0 \\ 1&0&0 \\ 0&0&1 \end{pmatrix}}$  \\[2em] \hline  &&&& \\[-.8em]
$D_{12}$ & $W_8$ & $\left<g_1^{\,3},g_2g_1\right>$ & $\begin{array}{c}  g_1^{\,3} \cdot 1 =-1 \\ g_2g_1\cdot 1 =1\end{array}$ & \footnotesize{$\begin{pmatrix} 1&0&0 \\ 0&0&-1 \\ 0&1&0 \end{pmatrix}$} \\[2em] \hline  &&&& \\[-.8em]
$A_4$ & $W_9$ & $\left<g_2,g_1^{-1}g_2g_1\right>$ & $\begin{array}{c} g_2\cdot 1 =1 \\ g_1^{-1}g_2g_1 \cdot 1 = -1 \end{array}$ & \footnotesize{$\begin{pmatrix} 0&1&0 \\ 1&0&0 \\ 0&0&1 \end{pmatrix}$}  \\[2em] \hline  &&&& \\[-.8em]
$A_4 \times \mathbb{Z}_2$ & $W_1$ & $\left<g_2,\, g_1^{-1}g_2g_1,\, g_3\right>$ & $\begin{array}{c} g_2 \cdot 1=1 \\ g_1^{-1}g_2g_1 \cdot 1=-1 \\ g_3\cdot 1 = -1 \end{array}$ & \footnotesize{$\begin{pmatrix} 0&1&0 \\ 1&0&0 \\ 0&0&1 \end{pmatrix}$}  \\[2em] \hline  &&&& \\[-.8em]
\multirow{4}*{$S_4$} & $W_6$ & $\left<g_1,\, g_1^{\,3}g_2g_1^{\,2}g_2\right>$ & $\begin{array}{c}g_1\cdot 1 =1 \\ g_1^{\,3}g_2g_1^{\,2}g_2 \, \cdot 1 =-1 \end{array}$ & \footnotesize{$\begin{pmatrix} 0&1&0 \\ 1&0&0 \\ 0&0&1\end{pmatrix}$} \\[2em] \cline{2-5} \\[-.8em]
& $W_7$ & $\left<g_1, \, g_1^{\,3}g_2g_1^{\,2}g_2\right>$ & $\begin{array}{c} g_1\cdot 1 = -1 \\ g_1^{\,3}g_2g_1^{\,2}g_2 \, \cdot 1 = 1 \end{array}$ & \footnotesize{$\begin{pmatrix} 0&1&0 \\ 1&0&0 \\ 0&0&1\end{pmatrix}$}  \\[2em] \hline  &&&& \\[-.8em] 
$S_4 \times \mathbb{Z}_2$ & $W_1$ & $\left<g_1, \, g_1^{\,3}g_2g_1^{\,2}g_2, \,g_3\right>$ & $\begin{array}{c} g_1\cdot 1 =1 \\ g_1^{\,3}g_2g_1^{\,2}g_2 \, \cdot 1 =-1 \\ g_3 \cdot 1 =-1 \end{array}$ & \footnotesize{$\begin{pmatrix} 0&1&0 \\ 1&0&0 \\ 0&0&1\end{pmatrix}$}  \\[2em] \hline
\end{tabular}}
\caption{Conjugacy classes in $\text{GL}(3,\mathbb{Z})$ that can be applied Theorem \ref{induced 1}}
\end{table}

\noindent (2) The cases of Theorem \ref{induced 2} and Corollary \ref{cor of induced 2}

First, consider when $G$ is a subgroup of $\text{GL}(2,\mathbb{Z})$. Table 5 summarize a conjugacy class of $G$, a subgroup $N$ of $G$, a $G$-integer $\mathcal{U}$ and $T\in \text{GL}(3,\mathbb{Z})$ that satisfy the similarity condition of Theorem \ref{induced 2}. For each case, the set $\mathcal{B}=\{e_1=1\otimes 1 , \, e_2=g_1\otimes 1, \, e_3= g_1^{\,2}\otimes 1\}$ is a $\mathbb{Z}$-basis for $\ind{N}{G}{\mathcal{U}}$. By considering $T_{g_i}$ as its matrix representation with respect to $\mathcal{B}$, we can check $T^{-1}T_{g_i}T$ satisfy the similarity condition of Theorem \ref{induced 2}. The cases $G=D_6^{\prime},C_3$ can be computed by Theorem \ref{induced 2}-(2) and the cases $G=D_{12},D_6,C_6$ can be computed by Corollary \ref{cor of induced 2}.

\begin{table}[ht]
\footnotesize{
\begin{tabular}{c|c|c|c|c|c|c|c}
\hline
$G$ & $N$ & $\mathcal{U}$ & $T$ & $G$ & $N$ & $\mathcal{U}$ & $T$ \\ \hline \hline 
&&&&&&& \\[-.8em]
$D_{12}$ & $\left<g_1^{\,3}, \, g_2\right>$ & $g_1\cdot 1 =-1, \, g_2\cdot 1 =1$ & \footnotesize{$\begin{pmatrix} 1&1&1 \\ 0&0&1 \\ 0&-1&0 \end{pmatrix}$} & $C_6$ & $\left<g_1^{\,3}\right>$ & $g_1 \cdot 1 =-1$ & \footnotesize{$\begin{pmatrix} 1&1&1 \\ 0&0&1 \\ 0&-1&0 \end{pmatrix}$} \\[2em] \hline
&&&&&&& \\[-.8em]
$D_6$ & $\left<g_2g_1\right>$ & $g_1 \cdot 1 =1,\, g_2 \cdot 1 =-1$ & \footnotesize{$\begin{pmatrix} 1&0&1 \\ 0&-1&-1 \\ 0&1&0 \end{pmatrix}$} & $C_3$ & trivial & $g_1 \cdot 1 = 1$ & \footnotesize{$\begin{pmatrix} 1&1&1 \\ 0&-1&0 \\ 0&0&-1 \end{pmatrix}$}  \\[2em] \hline
&&& \\[-.8em]
$D_6^{\prime}$ & $\left<g_2\right>$ & $g_1 \cdot 1 = 1,\, g_2 \cdot 1 = 1$ & \footnotesize{$\begin{pmatrix} 1&1&1 \\ 0&-1&0 \\ 0&0&-1 \end{pmatrix}$} \\[2em] \hline
\end{tabular}}
\caption{Conjugacy classes in $\text{GL}(2,\mathbb{Z})$ that can be applied Theorme \ref{induced 2} and Corollary \ref{cor of induced 2}}
\end{table}

Next, consider when $G$ is a subgroup of $\text{GL}(3,\mathbb{Z})$. Only for $G\simeq D_8$, redefine $g_1=A$ and $g_2=BA$, where $A$ is the first generator of $G$ and $B$ the second generator of $G$ appeared in Tahara's list \cite{tahara}. For example, when $G\simeq D_8$ and its conjugacy class is $W_{11}$, $g_1$ and $g_2$ are given as follows, see Tahara\cite{tahara}.
\[
g_1 = \footnotesize{\begin{pmatrix} 1&0&1 \\ 0&0&-1 \\ 0&1&0 \end{pmatrix}},\,\, g_2 = \begin{pmatrix} -1&0&0 \\ 0&0&-1 \\ 0&-1&0 \end{pmatrix} \begin{pmatrix}1&0&1 \\ 0&0&-1 \\ 0&1&0 \end{pmatrix} = \begin{pmatrix} -1&0&-1 \\ 0&-1&0 \\ 0&0&1 \end{pmatrix}
\] 
Table 6 summarize the subgroup $N$ of $G$, an $N$-integer $\mathcal{U}$ and $T\in \text{GL}(4,\mathbb{Z})$ that satisfy the similarity condition of Theorem \ref{induced 2}.

Let $\mathcal{B}$ denote the $\mathbb{Z}$-basis for $\ind{N}{G}{\mathcal{U}}$. For $G=\mathbb{Z}_2\times \mathbb{Z}_2,\, \mathbb{Z}_2\times \mathbb{Z}_2 \times \mathbb{Z}_2$, we take $\mathcal{B} = \{e_1=1\otimes 1, \, e_2=g_1\otimes 1,\, e_3=g_2\otimes 1,\, e_4 = g_1g_2\otimes 1\}$. For $G=\mathbb{Z}_2 \times \mathbb{Z}_4,\, D_8$, we take $\mathcal{B} =\{e_1=1\otimes 1,\, e_2= g_1\otimes 1,\, e_3= g_1^{\,2}\otimes 1, \, e_4 = g_1^{\,3}\otimes 1 \}$. For $G=A_4,\, A_4\times \mathbb{Z}_2$, we take $\mathcal{B}= \{e_1=1\otimes 1,\, e_2=g_2\otimes 1, \, e_3 = g_1^{-1}g_2g_1 \otimes 1,\, e_4 = g_1g_2g_1^{-1}\otimes 1\}$. By considering $T_{g_i}$ as its matrix representation with respect to $\mathcal{B}$, we can check $T^{-1}T_{g_i}T$ saistfy the similarity condition of Theorem \ref{induced 2}. Theorem \ref{induced 2}-(2) can be applied to the following: $W_{12}$,$W_{13}$ of $G=\mathbb{Z}_2\times \mathbb{Z}_2,$, $W_{11}$ of $G=A_4$, $W_{12}$ and $W_{14}$ of $G=D_8$. And Corollary \ref{cor of induced 2} can be applied to the following: $W_2$ of $G=\mathbb{Z}_4\times \mathbb{Z}_2$, $W_5$ of $G=\mathbb{Z}_2\times \mathbb{Z}_2 \times \mathbb{Z}_2$, $W_{11}$ and $W_{13}$ of $G=D_8$, $W_{3}$ of $G=A_4 \times \mathbb{Z}_2$. \newpage

\begin{table}[ht]
\footnotesize{
\begin{tabular}{c|c|c|c|c}
\hline
$G$ & conj class & $N$ & $\mathcal{U}$ & $T$ \\ \hline \hline 
\multirow{4}*{$\mathbb{Z}_2 \times \mathbb{Z}_2$} & $W_{12}$ & trivial & $\begin{array}{c} g_1 \cdot 1 =1 \\ g_2\cdot 1 =1 \end{array}$ & \footnotesize{$\begin{pmatrix} 1&1&1&0 \\ 0&-1&0&1 \\ 0&0&0&-1 \\ 0&0&-1&0  \end{pmatrix}$} \\ \cline{2-5}
& $W_{13}$ & trivial & $\begin{array}{c} g_1 \cdot 1 =1 \\ g_2 \cdot 1 =-1 \end{array}$ & \footnotesize{$\begin{pmatrix} 1&1&1&0 \\ 0&-1&0&1 \\ 0&0&0&1 \\ 0&0&1&0 \end{pmatrix}$}\\ \hline
$\mathbb{Z}_4\times \mathbb{Z}_2$ & $W_2$ & $\left<g_2\right>$ & $\begin{array}{c} g_1\cdot 1 = -1 \\ g_2 \cdot 1 =-1 \end{array}$ & \footnotesize{$\begin{pmatrix} 1&1&0&0 \\ 0&1&1&0 \\ 0&1&1&1 \\ 0&1&0&1 \end{pmatrix}$} \\ \hline
$\mathbb{Z}_2\times \mathbb{Z}_2 \times \mathbb{Z}_2$ & $W_5$ & $\left<g_3\right>$ & $\begin{array}{c} g_1\cdot 1 = 1 \\ g_2\cdot 1 =1 \\ g_3\cdot 1 = -1 \end{array}$ & \footnotesize{$\begin{pmatrix} 1&1&1&0 \\ 0&-1&0&1 \\ 0&0&0&-1 \\ 0&0&-1&0  \end{pmatrix}$} \\ \hline 
\multirow{12}*{$D_8$} & $W_{11}$ & $\left<g_2g_1\right>$ & $\begin{array}{c} g_1\cdot 1 = -1 \\ g_2 \cdot 1 =1 \end{array}$ & \footnotesize{$\begin{pmatrix} 1&1&0&1 \\ 0&1&0&0 \\ 0&1&1&0 \\ 0&1&1&1 \end{pmatrix}$} \\ \cline{2-5}
& $W_{12}$ & $\left<g_2g_1\right>$ & $\begin{array}{c} g_1\cdot 1 =-1 \\ g_2\cdot 1 = -1 \end{array}$ & \footnotesize{$\begin{pmatrix} 1&1&0&1 \\ 0&1&0&0 \\ 0&1&1&0 \\ 0&1&1&1 \end{pmatrix}$} \\ \cline{2-5} 
& $W_{13}$ & $\left<g_2g_1\right>$ & $\begin{array}{c} g_1\cdot 1 = 1 \\ g_2\cdot 1 = -1 \end{array}$ & \footnotesize{$\begin{pmatrix} 1&1&0&1 \\ 0&-1&0&0 \\ 0&1&1&0 \\ 0&-1&-1&-1  \end{pmatrix}$} \\ \cline{2-5}
& $W_{14}$ & $\left<g_2g_1\right>$ & $\begin{array}{c} g_1 \cdot 1 =1 \\ g_2\cdot 1 =1 \end{array}$ & \footnotesize{$\begin{pmatrix} 1&1&0&1 \\ 0&-1&0&0 \\ 0&1&1&0 \\ 0&-1&-1&-1 \end{pmatrix}$} \\ \hline
$A_4$ & $W_{11}$ & $\left<g_1\right>$ & $\begin{array}{c} g_1\cdot 1 =1 \\ g_2\cdot 1 =1 \end{array}$ & \footnotesize{$\begin{pmatrix} 2&1&1&1 \\ -1&-1&0&0 \\ -1&0&-1&0 \\ -1&0&0&-1 \end{pmatrix}$} \\ \hline
$A_4 \times \mathbb{Z}_2$ & $W_3$ & $\left<g_1,g_3\right>$ & $\begin{array}{c} g_1\cdot 1 = 1 \\ g_2\cdot 1 =1 \\ g_3 \cdot 1 =-1 \end{array}$ & \footnotesize{$\begin{pmatrix} 2&1&1&1 \\ -1&-1&0&0 \\ -1&0&-1&0 \\ -1&0&0&-1 \end{pmatrix}$} \\ \hline
\end{tabular}}
\caption{Conjugacy classes in $\text{GL}(3,\mathbb{Z})$ that can be applied Theorem \ref{induced 2} and Corollary \ref{cor of induced 2}}
\end{table} \newpage

\noindent(3) The cases of Theorem \ref{main thm of 3.3}

First, consider when $G$ is a subgroup of $\text{GL}(2,\mathbb{Z})$. Since the form of generators of $D_4$, $D_2$, $C_2$, $C_1$ are given as (\ref{block similarity}), Theorem \ref{main thm of 3.3} holds for these cases. Next, consider when $G$ is a subgroup of $\text{GL}(3,\mathbb{Z})$. Here is the table for conjugacy classes of $\text{GL}(3,\mathbb{Z})$ that can be applied Theorem \ref{main thm of 3.3}.

\begin{table}[h]
\begin{center}
\footnotesize{
\begin{tabular}{c|c|c|c|c|c}
\hline
$G$ & conj class & $G$ & conj class & $G$ & conj class \\ \hline \hline
$\mathbb{Z}_2$ & $W_1 \sim W_5 $ & $\mathbb{Z}_3$ & $W_1$ & $\mathbb{Z}_4$ & $W_1 \sim W_4$ \\ \hline
$\mathbb{Z}_2 \times \mathbb{Z}_2$ & $W_5 \sim W_{11}$ & $\mathbb{Z}_6$ & $W_1 \sim W_3$ & $D_6$ & $W_5 \sim W_8$  \\ \hline
$\mathbb{Z}_4 \times \mathbb{Z}_2$ & $W_1$ & $\mathbb{Z}_2\times \mathbb{Z}_2 \times \mathbb{Z}_2$ & $W_3 \sim W_4$ & $D_8$ & $W_7 \sim W_{10}$ \\ \hline
$\mathbb{Z}_6 \times \mathbb{Z}_2$ & $W_1$ & $D_{12}$ & $W_2 \sim W_7$ & $D_8 \times \mathbb{Z}_2$ & $W_1 \sim W_2$ \\ \hline
$D_{12} \times \mathbb{Z}_2$ & $W_4 \sim W_5$ & & & \\\hline
\end{tabular}}
\caption{Conjugacy classes in $\text{GL}(3,\mathbb{Z})$ that can be applied Theorem \ref{main thm of 3.3}}
\end{center}
\end{table}

\subsection{The number of real forms of 3-dimensional split toric varieties}

Consider when a fan $\Sigma$ is in $\mathbb{Z}^3$, $K=\mathbb{C}$ and $k=\mathbb{R}$. Suppose that $\text{Aut}_{\Sigma}$ is a finite cyclic group or a dihedral group. By applying the algorithm (\ref{algorithm}), we can write $\coho{1}{G}{\text{Aut}_{\Sigma}^T}$ in terms of Brauer groups. Using $\brau{\mathbb{R}}{\mathbb{C}}\simeq \mathbb{Z}_2$, we get the number of $\mathbb{C}/\mathbb{R}$-forms of $X_{\Sigma}$ as follows. Denote this number by $\#$.

\begin{table}[ht]
\renewcommand{\arraystretch}{1.3}
\footnotesize{
\begin{tabular}{c|c|c|c|c|c|c|c|c}
\hline
$\text{Aut}_{\Sigma}$ & conj class & \#  & $\text{Aut}_{\Sigma}$ & conj class & \# & $\text{Aut}_{\Sigma}$ & conj class & $\#$  \\ \hline \hline
\multirow{4}*{$\mathbb{Z}_2$} & \multirow{1}*{$W_4$} & 2 & \multirow{2}*{$\mathbb{Z}_3$} & \multirow{1}*{$W_2$} & 2 & \multirow{4}*{$D_{12}$} & \multirow{1}*{$W_3$} & 5  \\ \cline{2-3} \cline{5-6} \cline{8-9}
& \multirow{1}*{$W_2, W_3$} & 3 & & \multirow{1}*{$W_1$} & 3 & & \multirow{1}*{$W_4, W_5$} & 6 \\ \cline{2-3} \cline{4-6} \cline{8-9}
& \multirow{1}*{$W_1$} & 5 & \multirow{2}*{$\mathbb{Z}_4$} & \multirow{1}*{$W_3, W_4$} & 3 & & \multirow{1}*{$W_2$} & 7 \\ \cline{2-3} \cline{5-6} \cline{8-9}
& \multirow{1}*{$W_5$} & 9 & & \multirow{1}*{$W_1, W_2$} & 4 & & \multirow{1}*{$W_6, W_7, W_8$} & 8 \\ \hline
\multirow{4}*{$\mathbb{Z}_2 \times \mathbb{Z}_2$} & \multirow{1}*{$W_{13}, W_{15}$} & 5 & \multirow{4}*{$\mathbb{Z}_2 \times \mathbb{Z}_2$} & \multirow{1}*{$W_7$} & 9 & \multirow{2}*{$\mathbb{Z}_6$} & \multirow{1}*{$W_1, W_2$} & 3 \\ \cline{2-3} \cline{5-6} \cline{8-9}
& \multirow{1}*{$W_9, W_{10}$} & 6 & & \multirow{1}*{$W_{11}$} & 10 & & \multirow{1}*{$W_3, W_4$} & 5 \\ \cline{2-3} \cline{5-6} \cline{7-9}
& \multirow{1}*{$W_{12}, W_{14}$} & 7 & & \multirow{1}*{$W_{6}$} & 13 & \multirow{2}*{$D_6$} & \multirow{1}*{$W_6, W_8, W_{10}$} & 2 \\ \cline{2-3} \cline{5-6} \cline{8-9}
& \multirow{1}*{$W_8$} & 8 & & \multirow{1}*{$W_5$} & 15 & & \multirow{1}*{$W_5, W_7, W_9$} & 3 \\ \hline
\multirow{2}*{$D_8$} & \multirow{1}*{$W_{12}$} & 5 & \multirow{2}*{$D_8$} & \multirow{1}*{$W_8, W_{11}$} & 7 & \multirow{2}*{$D_8$} & \multirow{1}*{$W_{10}$} & 9 \\ \cline{2-3} \cline{5-6} \cline{8-9} 
& \multirow{1}*{$W_{13},W_{14}$} & 6 & & \multirow{1}*{$W_9$} & 8 & & \multirow{1}*{$W_{7}$} & 10 \\\hline
\end{tabular}}
\end{table}

\section*{Appendix B} 
The table of Theorem \ref{D_6} was made through the algorithm (\ref{algorithm}). In Appendix B, we summarize the step (1) and (2) of the algorithm (\ref{algorithm}) in this theorem. These steps can be completed as Examples of Section 3. For each $g\in G$, let $T_g$ be a linear operator defined as (\ref{def:operator}). And we follow the notation $L$, $\overline{G}$, $\overline{\varphi}$ in (\ref{algorithm}) and assume $\overline{G}$ is a subgroup of $\text{GL}(3,\mathbb{Z})$ via $\overline{\varphi}$.

\begin{lem}\label{lem 1}
Suppose that $\overline{G}=\operatorname{Gal}(L/k)\simeq \mathbb{Z}_2 \leq \operatorname{GL}(3,\mathbb{Z})$ and let $s$ be a generator of $\overline{G}$. Then,
\end{lem}
\begin{center}
\begin{table}[ht]
\footnotesize{
\begin{tabular}{c|c|c|c|c|c}
\hline
\multirow{2}*{generator of $\overline{G}$} & \multirow{2}*{$\coho{1}{\overline{G}}{{}_{\overline{\varphi}}\mathcal{T}(L)}$} & generators & \multirow{2}*{generator of $\overline{G}$} & \multirow{2}*{$\coho{1}{\overline{G}}{{}_{\overline{\varphi}}\mathcal{T}(L)}$} & generators \\
& & of cohomology & & &of cohomology \\ \hline \hline
&&&&& \\[-.8em]
$\begin{pmatrix} -1&0&0 \\ 0&0&-1 \\ 0&-1&0 \end{pmatrix}$ & \multirow{7}*{$\brau{k}{L}$} & \multirow{6}*{$c(a)_s = (a,1,1)$} & $\begin{pmatrix} 1&0&0 \\ 0&0&1 \\ 0&1&0 \end{pmatrix}$ & \multirow{7}*{trivial} & \multirow{7}*{trivial} \\[2em] \cline{0-0} \cline{4-4}
&&&&& \\[-.8em]
$\begin{pmatrix} -1&0&0 \\ 0&0&1 \\ 0&1&0 \end{pmatrix}$ & & \multirow{2}*{for each $a \in k^\times$} & $\begin{pmatrix} 1&0&0 \\ 0&0&-1 \\ 0&-1&0  \end{pmatrix}$ & & \\[2em] \cline{0-0} \cline{4-4}
&&&&& \\[-.8em] 
$\begin{pmatrix} 0&0&-1 \\ 0&-1&0 \\ -1&0&0 \end{pmatrix}$ & & & $\begin{pmatrix} 0&0&1 \\ 0&1&0 \\ 1&0&0 \end{pmatrix}$ & & \\[2em] \hline
\end{tabular}}
\end{table}
\end{center}
\vspace{-2em}
\begin{proof}
Consider when $s$ is given as follows with submatrices $A$ and $B$.
\[
\footnotesize{ s=\begin{pmatrix}-1 &0&0 \\ 0&0&1 \\ 0&1&0 \end{pmatrix},\, A=-1, \,\,B=\begin{pmatrix} 0&1 \\ 1&0 \end{pmatrix}}
\]
Let $\pi:\overline{G} \rightarrow \text{Aut}(\mathbb{Z})$ be a group homomorphism defined by $\pi(s)=A$ and $\psi:\overline{G} \rightarrow \text{GL}(2,\mathbb{Z})$ be a group homomorphism defined by $\psi(s)=B$. By Theorem \ref{main thm of 3.3}-(1), we have $\coho{1}{\overline{G}}{{}_{\overline{\varphi}}\mathcal{T}(L)} \simeq \coho{1}{\overline{G}}{{}_{\pi}\mathcal{T}(L)} \oplus \coho{1}{\overline{G}}{{}_{\psi}\mathcal{T}(L)}$. Since $A$ is a diagonal matrix, we get $\coho{1}{\overline{G}}{{}_{\pi}\mathcal{T}(L)} \simeq \brau{k}{L}$ by Theorem \ref{main thm of 3.3}-(3). Via the injective map $\psi$, $\overline{G}$ is also considered as a subgroup of $\text{GL}(2,\mathbb{Z})$. We use Theorem \ref{induced 1} to compute $\coho{1}{\overline{G}}{{}_{\psi}\mathcal{T}(L)}$. Take a trivial subgroup $N$ of $\overline{G}$ and $\mathcal{U}$ be a $N$-integer with a trivial action. Since a matrix representation of $T_s$ with respect to $\{e_1=1\otimes1, \, e_2= s\otimes 1\}$ is same as $B$, we have $\coho{1}{\overline{G}}{{}_{\psi}\mathcal{T}(L)} \simeq \brau{L}{L}\simeq 1$. Applying Remark \ref{rmk}-(1) and (3), generators in the table can be obtained. Since every matrices of the left side are similar, same analysis holds.

Consider when $s$ is given as follows with submatrices $A$ and $B$.
\[\footnotesize{
s= \begin{pmatrix}1&0&0 \\ 0&0&1 \\ 0&1&0 \end{pmatrix},\,\, A=1,\,\, B=\begin{pmatrix}0&1 \\ 1&0 \end{pmatrix}}
\]
By Theorem \ref{main thm of 3.3}-(2), $\coho{1}{\overline{G}}{{}_{\overline{\varphi}}\mathcal{T}(L)} \simeq \coho{1}{\overline{G}}{{}_{\psi}\mathcal{T}(L)}$ and the last term is trivial as the previous paragraph. Since every matrices of the right side are similar, same analysis holds.
\end{proof} 

\begin{lem}\label{lem 2}
Suppose that $\overline{G}=\operatorname{Gal}(L/k)\simeq \mathbb{Z}_3 \leq \operatorname{GL}(3,\mathbb{Z})$ and let $r$ be a generator of $\overline{G}$. Then,
\end{lem}
\begin{center}
\begin{table}[ht]
\hspace*{-3em}
\footnotesize{
\begin{tabular}{c|c|c|c|c|c}
\hline
\multirow{2}*{generator of $\overline{G}$} & \multirow{2}*{$\coho{1}{\overline{G}}{{}_{\overline{\varphi}}\mathcal{T}(L)}$} & generators & \multirow{2}*{generator of $\overline{G}$} & \multirow{2}*{$\coho{1}{\overline{G}}{{}_{\overline{\varphi}}\mathcal{T}(L)}$} & generators \\
& & of cohomology & & & of cohomology \\ \hline \hline
\multirow{4}*{$\begin{pmatrix}1&0&0 \\ 0&0&-1 \\ 0&1&-1\end{pmatrix}$} & \multirow{4}*{$\brau{k}{L}$} & \multirow{3}*{$c(a)_r = (1,1,a^{-1})$} & \multirow{4}*{$\begin{pmatrix}0&1&0 \\ 0&0&1 \\ 1&0&0\end{pmatrix}$} & \multirow{4}*{trivial} & \multirow{4}*{trivial} \\ 
& & \multirow{3}*{for each $a \in k^\times$} & & & \\[2.5em] \hline
\end{tabular}}
\end{table}
\end{center}
\vspace{-2em}
\begin{proof}
Consider when $r$ is given as follows with submatrices $A$ and $B$.
\[\footnotesize{
r= \begin{pmatrix} 1&0&0 \\ 0&0&-1 \\ 0&1&-1 \end{pmatrix},\,\, A=1,\, B=\begin{pmatrix} 0&-1 \\ 1&-1 \end{pmatrix}}
\]
Let $\psi:\overline{G} \rightarrow \text{GL}(2,\mathbb{Z})$ be a group homomorphism defined by $\psi(s)=B$. By Theorem \ref{main thm of 3.3}-(2), we have $\coho{1}{\overline{G}}{{}_{\overline{\varphi}}\mathcal{T}(L)} \simeq \coho{1}{\overline{G}}{{}_{\psi}\mathcal{T}(L)}$. Via the injective map $\psi$, $\overline{G}$ is also considered as a subgroup of $\text{GL}(2,\mathbb{Z})$. Now, we use Theorem \ref{induced 2} to compute $\coho{1}{\overline{G}}{{}_{\psi}\mathcal{T}(L)}$. Take a trivial subgroup $N$ of $\overline{G}$ and a $\overline{G}$-integer $\mathcal{U}$ with a trivial action. A matrix representation of $T_r$ with respect to $\{e_1=r^i\otimes 1 \,|\, i=0,1,2\}$ is given as follows. By observing $T^{-1}T_rT$, we check the similarity condition of Theorem \ref{induced 2}.
\[\footnotesize{
T_r = \begin{pmatrix} 0&0&1 \\ 1&0&0 \\ 0&1&0 \end{pmatrix}, \, T=\begin{pmatrix} 1&-1&1 \\ 0&0&-1 \\ 0&1&0 \end{pmatrix},\, T^{-1}T_rT = \begin{pmatrix} 1&0&0 \\ 0&0&-1 \\ -1&1&-1 \end{pmatrix}}
\]
By Theorem \ref{induced 2}-(2), we get $\{c(a)\,|\,a\in k^\times\}$ as a collection of generators of $\coho{1}{\overline{G}}{{}_{\overline{\varphi}}\mathcal{T}(L)}$ where $c(a)_r = (1,a^{-1})$. A calculation shows that $c(a)$ is cohomologous to $c(a^{\prime})$ iff $a^{\prime}=a\norm{K}{k}{(x)}$ for some $x\in K^\times$. Therefore, the natural map $c(a) \mapsto \overline{a}$ gives the relation $\coho{1}{\overline{G}}{{}_{\psi}\mathcal{T}(L)}\simeq \brau{k}{L}$ where $\overline{\phantom{a}}$ denotes a representative of an element of the Brauer group. By applying Remark \ref{rmk}-(2), generators of $\coho{1}{\overline{G}}{{}_{\overline{\varphi}}\mathcal{T}(L)}$ in the table can be obtained.

Now, consider the second case. For this case, we use Theorem \ref{induced 1}. Take a trivial subgroup $N$ of $\overline{G}$ and an $N$-integer $\mathcal{U}$ with a trivial action. A matrix representation of $T_r$ with respect to $\{e_1=1\otimes 1 , e_2=r^2\otimes 1, e_3 = r\otimes 1\}$ is same as the second matrix. Hence, $\coho{1}{\overline{G}}{{}_{\overline{\varphi}}\mathcal{T}(L)}\simeq \brau{L}{L}$ is trivial.
\end{proof}

\begin{lem}\label{lem 3}
Let $W_i$ be the set of generators in Table 1. Suppose that $\overline{G}=\operatorname{Gal}(L/k)\simeq D_6 \leq \operatorname{GL}(3,\mathbb{Z})$. When $\overline{G}$ is generated by $W_i$, $r$ denotes the first matrix of $W_i$ and and $s$ denotes the second matrix of $W_i$. Then, we have
\begin{center}
\begin{table}[ht]
\footnotesize{
\begin{tabular}{c|c|c|c|c|c}
\hline
$\operatorname{conj}$ $\operatorname{class}$ & $\coho{1}{\overline{G}}{{}_{\overline{\varphi}}\mathcal{T}(L)}$ & $\operatorname{conj}$ $\operatorname{class}$ & $\coho{1}{\overline{G}}{{}_{\overline{\varphi}}\mathcal{T}(L)}$ & $\operatorname{conj}$ $\operatorname{class}$ & $\coho{1}{\overline{G}}{{}_{\overline{\varphi}}\mathcal{T}(L)}$ \\ \hline \hline
$W_5$ & $\operatorname{Br}(k|L^{\left<r\right>}) \oplus \operatorname{Br}(k|L^{\left<s\right>})$ & $W_6$ & $M$ & $W_7$ & $\operatorname{Br}(k|L^{\left<r\right>})\oplus M$ \\ \hline
$W_8$ & $\operatorname{Br}(k|L^{\left<s\right>})$ & $W_9$ & $\operatorname{Br}(L^{\left<s\right>}|L)$ & $W_{10}$ & $\operatorname{trivial}$ \\ \hline  
\end{tabular}}
\end{table}
\end{center} \vspace{-2em}
where $M = \displaystyle{\frac{\{a\in L^\times \, | \, r(a)=a,s(a)=a^{-1}\}}{\{\norm{L}{L^{\left<r\right>}}{(b)}\,|\,b \in L^\times \,\, \text{with} \,\, \norm{L}{L^{\left<sr\right>}}{(b)}=1\}}}$.
\end{lem}
\begin{proof}
When $\overline{G}$ is generated by $W_{10}$, we showed $\coho{1}{\overline{G}}{{}_{\overline{\varphi}}\mathcal{T}(L)}\simeq \brau{L^{\left<s\right>}}{L}$ in Example \ref{induced 1 exa}. Consider when $\overline{G}$ is generated by $W_9$. As in the case of $W_{10}$, we can show $\coho{1}{\overline{G}}{{}_{\overline{\varphi}}\mathcal{T}(L)}$ is trivial by using Theorem \ref{induced 1} in the setting of $N=\{1,s\}$ and an $N$-integer $\mathcal{U}$ with a trivial action. When $\overline{G}$ is generated by $W_7$, we showed $\coho{1}{\overline{G}}{{}_{\overline{\varphi}}\mathcal{T}(L)}\simeq \brau{k}{L^{\left<r\right>}}\oplus M$ in Example \ref{exa2 of 3.3}. Consider when $\overline{G}$ is generated by $W_6$. By Theorem \ref{main thm of 3.3}-(2), $\coho{1}{\overline{G}}{{}_{\overline{\varphi}}\mathcal{T}(L)}$ reduced to the Galois cohomology in Example \ref{induced 2 exa 2}, so it is isomorphic to $M$. Now, consider when $\overline{G}$ is generated by $W_5$. Let $A_1$, $A_2$, $B_1$ and $B_2$ be submatrices of $r$ and $s$ as follows. 
\[\footnotesize{
A_1=1,\,\, A_2=-1,\,\, B_1=\begin{pmatrix} 0& -1 \\ 1&-1 \end{pmatrix}, \,\, B_2 = \begin{pmatrix} 0&-1 \\ -1&0 \end{pmatrix}}
\]
Let $\pi:\overline{G} \rightarrow \text{Aut}(\mathbb{Z})=\{1,-1\}$ be a group homomorphism that sends $r$ to $A_1$ and $s$ to $A_2$ and $\psi: \overline{G} \rightarrow \text{GL}(2,\mathbb{Z})$ be a group homomorphism that sends $r$ to $B_1$ and $s$ to $B_2$. By Theorem \ref{main thm of 3.3}, we have $\coho{1}{\overline{G}}{{}_{\overline{\varphi}}\mathcal{T}(L)} \simeq \coho{1}{\overline{G}}{{}_{\pi}\mathcal{T}(L)}\oplus \coho{1}{\overline{G}}{{}_{\psi}\mathcal{T}(L)} \simeq \brau{k}{L^{\left<r\right>}}\oplus \coho{1}{\overline{G}}{{}_{\psi}\mathcal{T}(L)}$. Now, we compute $\coho{1}{\overline{G}}{{}_{\psi}\mathcal{T}(L)}$ by Theorem \ref{induced 2}. Let $N=\{1,s\}$ be a subgroup of $\overline{G}$ and $\mathcal{U}$ be a $\overline{G}$-integer with a trivial action. Matrix representations of $T_r$ and $T_s$ with respect to $\{e_i=r^{i-1}\otimes 1 \, | \, i=1,2,3\}$ are given as follows.
\[\footnotesize{
T_r = \begin{pmatrix}0&0&1 \\ 1&0&0 \\ 0&1&0 \end{pmatrix},\, T_s = \begin{pmatrix} 1&0&0 \\ 0&0&1 \\ 0&1&0\end{pmatrix},\,T=\begin{pmatrix} 1&-1&1 \\ 0&0&-1 \\ 0&1&0 \end{pmatrix}}
\]
By observing the forms of $T^{-1}T_rT$ and $T^{-1}T_sT$, we check the similarity condition of Theorem \ref{induced 2}.
\[\footnotesize{
T^{-1}T_rT = \begin{pmatrix}1&0&0 \\ 0&0&-1 \\ -1&1&-1 \end{pmatrix}, T^{-1}T_sT = \begin{pmatrix}1&0&0 \\ 0&0&-1 \\ 0&-1&0 \end{pmatrix}}
\]
By Theorem \ref{induced 2}-(2), we get $\{c(a)\,|\, a \in k^\times\}$ as a collection of generators of $\coho{1}{\overline{G}}{{}_{\psi}\mathcal{T}(L)}$ where $c(a)_r=(1,a^{-1})$ and $c(a)_s=(1,1)$. A calculation shows that $c(a)$ is cohomologous $c(a^{\prime})$ iff $a^{\prime}=a\norm{K^{\left<s\right>}}{k}{(x)}$ for some $x \in (K^{\left<s\right>})^\times$. Therefore, the natural map $c(a) \mapsto \overline{a}$ gives $\coho{1}{\overline{G}}{{}_{\varphi}\mathcal{T}(K)} \simeq \brau{k}{K}$ where $\overline{\phantom{a}}$ denotes a representative of an element in the Brauer group. Hence, $\coho{1}{\overline{G}}{{}_{\overline{\varphi}}\mathcal{T}(L)} \simeq \brau{k}{L^{\left<r\right>}}\oplus \brau{k}{L^{\left<s\right>}}$. The similar analysis can be applied when $\overline{G}$ is generated by $W_7$.
\end{proof}

\bibliographystyle{plain}
\bibliography{reference}

\end{document}